\documentclass{amsart}
\usepackage[utf8]{inputenc}
\usepackage[T1]{fontenc}

\usepackage[cal=dutchcal]{mathalfa}

\usepackage[english]{babel}

\usepackage{libertine}
\usepackage{graphicx}
\usepackage{floatflt}
\usepackage{blindtext}
\usepackage{enumitem}
\usepackage{amsthm}
\usepackage{subfig}
\usepackage{listings}
\usepackage{listingsutf8}
\usepackage{amsmath}
\usepackage{framed}
\usepackage{minibox}
\usepackage{float}
\usepackage{wrapfig}
\usepackage{longtable}
\usepackage[strict]{changepage}
\usepackage{pgfplots}
\usepackage{nicefrac}
\usepackage{units}
\usepackage{etoolbox,tikz}
\usepackage{bbm}
\usepackage{hyperref}
\usepackage[autostyle]{csquotes}

\title{Optimal Synthesis on the Special Linear Group}

 \author[Lau]{Marcus T. C. Lau}
  \address{Marcus T. C. Lau;
 {Department of Mathematics, University of Michigan, 530 Church St, Ann Arbor, MI 48109, United States
 \\
\href{mtclau@umich.edu}{mtclau@umich.edu} }
 }

 \author[A. Bravo-Doddoli]{Alejandro Bravo-Doddoli}
  \address{Alejandro Bravo-Doddoli;
 {Department of Mathematics, University of Michigan, 530 Church St, Ann Arbor, MI 48109, United States
 \\
\href{abravodo@umich.edu}{abravodo@umich.edu}}
 }

\usepackage[autostyle]{csquotes}

\usetikzlibrary{external}
\usetikzlibrary{cd}

\AtBeginEnvironment{tikzcd}{\tikzexternaldisable}
\AtEndEnvironment{tikzcd}{\tikzexternalenable}

\usepackage{amsfonts}

\usepackage{amsthm}
\usepackage {amssymb}
\usepackage{pb-diagram}
\usepackage{tikz-cd}

\usepackage{anyfontsize}
\usepackage{mathtools}
\usepackage{amssymb}

\usetikzlibrary{matrix}
\pgfplotsset{width=11cm,compat=1.9}
\usepgfplotslibrary{external}
\newtheorem{Theorem}{Theorem}[section]

\newtheorem{Prop}[Theorem]{Proposition}
\newtheorem{Cor}[Theorem]{Corollary}
\newtheorem{lemma}[Theorem]{Lemma}
\newtheorem{Remark}[Theorem]{Remark}
\theoremstyle{plain}
\numberwithin{figure}{section}

\newtheorem*{backgroundtheorem}{Hadamard's Global Diffeomorphism Theorem}

\makeatletter
\newcommand*{\relrelbarsep}{.386ex}
\newcommand*{\relrelbar}{%
  \mathrel{%
    \mathpalette\@relrelbar\relrelbarsep
  }%
}
\newcommand*{\@relrelbar}[2]{%
  \raise#2\hbox to 0pt{$\m@th#1\relbar$\hss}%
  \lower#2\hbox{$\m@th#1\relbar$}%
}
\providecommand*{\rightrightarrowsfill@}{%
  \arrowfill@\relrelbar\relrelbar\rightrightarrows
}
\providecommand*{\leftleftarrowsfill@}{%
  \arrowfill@\leftleftarrows\relrelbar\relrelbar
}
\providecommand*{\xrightrightarrows}[2][]{%
  \ext@arrow 0359\rightrightarrowsfill@{#1}{#2}%
}
\providecommand*{\xleftleftarrows}[2][]{%
  \ext@arrow 3095\leftleftarrowsfill@{#1}{#2}%
}
\makeatother

\makeatletter
\providecommand*{\twoheadrightarrowfill@}{%
  \arrowfill@\relbar\relbar\twoheadrightarrow
}
\providecommand*{\twoheadleftarrowfill@}{%
  \arrowfill@\twoheadleftarrow\relbar\relbar
}
\providecommand*{\xtwoheadrightarrow}[2][]{%
  \ext@arrow 0579\twoheadrightarrowfill@{#1}{#2}%
}
\providecommand*{\xtwoheadleftarrow}[2][]{%
  \ext@arrow 5097\twoheadleftarrowfill@{#1}{#2}%
}
\makeatother

\makeatletter
\newcommand{\bigboxtimes}{%
  \mathop{\vphantom{\bigotimes}\mathpalette\bigbox@times\relax}%
  \slimits@
}
\newcommand{\bigbox@times}[2]{%
  \vcenter{%
    \hbox{%
      \scalebox{2}{$\m@th#1\boxtimes$}%
    }%
  }%
}
\makeatother

\makeatletter
\newcommand\niton{\mathrel{\m@th\mathpalette\canc@l\owns}}
\newcommand\canc@l[2]{{\ooalign{$\hfil#1/\mkern1mu\hfil$\crcr$#1#2$}}}
\makeatother



\makeatletter
\newcommand{\vas}{\bBigg@{3}}
\newcommand{\Vas}{\bBigg@{3.5}}
\newcommand{\vast}{\bBigg@{4}}
\newcommand{\Vast}{\bBigg@{5}}
\newcommand{\vastt}{\bBigg@{6}}
\newcommand{\Vastt}{\bBigg@{7}}
\makeatother

\DeclareMathOperator{\Exp}{Exp}

\DeclareMathOperator{\Ad}{Ad}
\DeclareMathOperator{\ad}{ad}

\DeclareMathOperator{\Id}{\operatorname{Id}}

\newcommand{\SLg}{\mathrm{SL}_2(\mathbb{R})}
\newcommand{\sla}{\mathfrak{sl}_2(\mathbb{R})}
\newcommand{\Lag}{\mathfrak{g}}
\newcommand{\MAS}{\mathcal{MaxStr}}
\newcommand{\MD}{\mathcal{M}}
\newcommand{\ND}{\mathcal{N}}
\newcommand{\F}{\operatorname{F}}
\newcommand{\Pm}{\operatorname{P}}
\newcommand{\Sym}{\operatorname{Sym}_1^+(2)}
\newcommand{\Tr}{\operatorname{tr}}
\newcommand{\T}{\mathrm{T}}
\newcommand{\SO}{\mathrm{SO}(2)}
\newcommand{\Tcut}{t_{\mathrm{cut}}}
\newcommand{\Tcon}{t_{\mathrm{con}}}
\newcommand{\Tmax}{t_{\mathrm{Max}}}

\def\R{\mathbb{R}}

\begin{document}

\begin{abstract}
We study the exponential map on the group of real two-by-two matrices with determinant one, endowed with a left-invariant Riemannian metric. Using the Euler-Poincar\'e equation and its linearization, we determine the conjugate and Maxwell times and identify a domain on which the exponential map is a global diffeomorphism onto the complement of the Maxwell stratum. This provides an alternative proof that the first Maxwell time coincides with the cut time. We obtain an explicit matrix description of the cut locus and characterize the globally minimizing geodesics. Furthermore, we analyze the first Maxwell-time function, determine its unique global minimum, and show that the injectivity radius is $\frac{\pi \sqrt{14}}{4}$. Finally, we characterize the intersection of the cut and conjugate loci and prove that every geodesic is equi-optimal.
\end{abstract}

\maketitle


\section{Introduction}\label{sec:int} 

In this paper, we study the exponential map on the special linear group $\SLg$  endowed with a left-invariant Riemannian metric. By left invariance, it suffices to study the exponential map at the identity element. The Riemannian problem considered here is the special case \(I_1=I_2=I_3\) of the family of symmetric left-invariant metrics studied by Podobryaev and Sachkov \cite{PodobryaevSachkov2018}, who used a Hamiltonian approach to determine the first Maxwell time, cut time, and cut locus. 

n contrast, we use a Lagrangian approach based on the Euler–Poincaré equation and its linearization. To make our treatment self-contained, we provide independent proofs of these results for the metric considered here. In particular, we identify a domain on which the exponential map is a global diffeomorphism onto the complement of the Maxwell stratum; see \ref{thm:main-3}. This establishes, through a different approach, that the first Maxwell time coincides with the cut time.

To the best of our knowledge, the following contributions are new. We obtain a direct matrix description of the Maxwell stratum, and hence of the cut locus, and identify the globally minimizing geodesics; see Lemma \ref{lem:g-ome-sym} and Corollary \ref{cor:cut-time}. In addition, we provide a detailed analysis of the first Maxwell-time function, determine its unique global minimum, compute the injectivity radius explicitly, characterize the intersection of the cut and conjugate loci, and prove that every geodesic is equi-optimal; see Proposition \ref{prp:f-max-t-reg} and Corollaries \ref{cor:rad-in}-\ref{cor:equi-opt}.


%

Our approach treats the exponential map as an optimal control problem and uses optimal synthesis to prove Theorem \ref{thm:main-3}. Optimal synthesis is a technique in geometric control theory that exploits the symmetries of the exponential map to identify Maxwell points, where distinct geodesics of equal length intersect \cite{jurdjevic1997geometric,agrachev2013control}. The corresponding Maxwell stratum provides a candidate for the cut locus. To verify that the Maxwell stratum is indeed the cut locus, we apply Hadamard's global diffeomorphism theorem \cite{krantz2002implicit} to a restriction of the exponential map to the domain corresponding to geodesic segments before their first Maxwell time. We show that this restriction is a global diffeomorphism onto the complement of the Maxwell stratum. 

We now introduce conjugate points, Maxwell points, and cut points, together with their associated sets, and present our main results.

   The special linear group $\SLg$ is a three-dimensional connected semisimple Lie group consisting of real $2\times2$ matrices with determinant one.
    A subgroup that plays a central role in our analysis is the special orthogonal group $\mathrm{SO}(2)$, consisting of orientation-preserving orthogonal matrices. Recall that the Lie algebra $\sla$ consists of traceless  $2\times2$ real matrices, while its subalgebra $\mathfrak{so}(2)$ consists of skew-symmetric matrices. We equip $\SLg$ with the left-invariant Riemannian metric induced by the inner product on $\sla$ defined by
    $$\langle \xi,\eta \rangle  \coloneqq \frac{1}{2}\Tr(\xi\eta^\T) , \text{ for } \xi,\eta \in \sla.$$
    This inner product induces a norm and a unit sphere on $\sla$, denoted by \(\|\cdot\|\) and $\mathbb{S}^2(\sla)$, respectively. For simplicity, we set  $\mathcal{E} \coloneqq (0,\infty)\times \mathbb{S}^2(\sla) $.

    Since every left-invariant Riemannian metric on a connected Lie group is geodesically complete \cite{Helgason1978}, the Hopf–Rinow theorem implies that the Riemannian exponential map $\Exp\vcentcolon\sla\to \SLg$ is defined on the entire Lie algebra and is surjective. Recall that $\gamma(t,\xi)$ denotes the geodesic satisfying the initial condition
    $$ \gamma(0,\xi) = \Id,\qquad \text{and} \qquad \dot{\gamma}(0,\xi) = \xi \in \sla.$$
    then 
    $$ \Exp(\xi) = \gamma(1,\xi). $$
    Since the exponential map is a diffeomorphism at the origin $0_2 \in \sla$, we focus on its behavior away from the origin. We therefore introduce the diffeomorphism 
    $$\Psi\vcentcolon\mathcal{E} \to \sla\setminus\{0_2\},\qquad \Psi(t,\hat{\xi}) = t\hat{\xi}. $$
    Using the homogeneity of the geodesic equation, we have $\gamma(t,\hat{\xi}) = \gamma(1,t\hat{\xi})$. Consequently, in these coordinates, the Riemannian exponential map takes the form
    $$ \Exp(t,\hat{\xi}) \coloneqq \Exp(\Psi(t,\hat{\xi})) = \Exp(t\hat{\xi}) = \gamma(t ,\hat{\xi}).$$


    A crucial step in the optimal synthesis is to study the symmetries of the exponential map.  A key symmetry in our analysis arises from the adjoint action of $\SO$ on $\sla$. The corresponding action on $\SLg$ defines an isometry $\Phi_Q(g) = Q gQ^{-1}$, with the property that $\Phi_Q(\Id) = \Id$ for all $Q\in \SO$. Then the exponential map satisfies 
    $$  (\Phi_Q \circ \Exp)(t,\xi) = \Exp(t,\Ad_Q(\xi)), $$
    for all $\xi\in \mathbb{S}^2(\sla)$, $Q\in \mathrm{SO}(2)$, and $ t \in [0,\infty)$. Let $\tilde{\xi}$ be a fixed unit vector spanning $\mathfrak{so}(2)$. This symmetry motivates us to define the following set, which is fundamental to our work:
    \begin{equation*}
    \Gamma_\omega  \coloneqq \left\{ \xi \in \mathbb{S}^2(\sla) \vcentcolon \omega\tilde{\xi} = \frac{1}{2}(\xi-\xi^\T) \right\}.
\end{equation*}
    As shown in Proposition \ref{prp:iso-action}, $\Gamma_\omega$ is invariant under the adjoint action of $\SO$, and this action is transitive on $\Gamma_\omega$. Consequently, all geodesics with initial velocities in $\Gamma_\omega$ are related by isometries fixing the identity. Therefore, to study their optimality properties, it suffices to consider a single geodesic $\gamma$ with initial velocity $\dot{\gamma}(0) \in \Gamma_\omega$.

    The first relevant set associated with the exponential map is the \textbf{conjugate locus}
    $$ \mathcal{Con} \coloneqq \Big\{ \Exp (t,\xi) \vcentcolon (t,\xi)\in\mathcal{E}  \;\text{and}\; d\Exp|_{(t,\xi)} \;\text{is singular} \Big\} .  $$
    If $(\Tcon ,\xi)\in  \mathcal{E}$ is a pair such that $\Exp(\Tcon ,\xi) \in \mathcal{Con}$, then  $\Exp(\Tcon ,\xi)$ is called a \textbf{conjugate point} along the geodesic $\gamma(t)\coloneqq \Exp(t,\xi)$ and $\Tcon $ is called a \textbf{conjugate time}. Recall that the conjugate locus has measure zero. Moreover, since the metric is real analytic, the conjugate times along each geodesic form a discrete set \cite{milnor1963morse}.

    By the preceding discussion and the transitivity of the $\SO$-action on $\Gamma_\omega$, every geodesic with initial velocity in $\Gamma_\omega$ has the same first conjugate time. In particular, 
   \begin{equation}\label{eq:con-Gamma}
       \Exp(\Tcon ,\Ad_Q(\xi)) \in \mathcal{Con},\;\;\text{for all}\;\;  Q\in \mathrm{SO}(2).
   \end{equation}
   This property motivates the following notation. Let $\Tcon (\omega)$ denote the \textbf{first conjugate time} along a geodesic $\gamma$ with initial velocity $\dot{\gamma}(0)\in \Gamma_\omega$. By convention, $\Tcon (\omega) = \infty$ if $\gamma$ has no conjugate points. With this notation, our first main theorem is as follows.
   \begin{Theorem}\label{thm:classic}
    Let $\xi \in \Gamma_\omega$. If $\gamma$ is a geodesic satisfying the initial conditions $\gamma(0) = \Id$ and $\dot{\gamma}(0) = \xi$, then 
    $$ \Tcon (\omega) = \begin{cases}
        \frac{\pi}{\sqrt{2\omega^2-1}}& \text{if}\; \omega \in [-1,-\frac{1}{\sqrt{2}})\cup (\frac{1}{\sqrt{2}},1 ], \\
        \infty & \text{if}\; \omega \in [-\frac{1}{\sqrt{2}},\frac{1}{\sqrt{2}} ]. \\
    \end{cases} $$
    Furthermore, the conjugate locus of the identity element is 
    $$ \mathcal{Con}(\Id) = \mathrm{SO}(2).  $$
     \end{Theorem}

    The second fundamental set for our study is the \textbf{set of Maxwell points}
    $$ \mathcal{Max}\coloneqq \{ \Exp(t,\xi_1) \vcentcolon (t,\xi_1),(t,\xi_2)\in \mathcal{E},\;\xi_1\neq \xi_2, \;\;\text{and}\;\;\Exp(t,\xi_1) = \Exp(t,\xi_2)\}.$$
    If $(\Tmax,\xi)\in  \mathcal{E}$ is a pair such that $\Exp(\Tmax,\xi) \in \mathcal{Max}$, then the point  $\Exp(\Tmax,\xi)$ is called a \textbf{Maxwell point} and $\Tmax$ is a \textbf{Maxwell time}.
    Since the geodesics are real analytic, if two distinct geodesics intersect at a set of times having a finite accumulation point, then they coincide identically. This would imply that their initial velocities coincide. Hence, for $\xi_1\neq \xi_2$, the set of times $t>0$ such that $\Exp(t,\xi_1) =  \Exp(t,\xi_2)$ is discrete. 

    In general, determining the set of all Maxwell points is difficult. One therefore typically restricts attention to Maxwell points arising from symmetries of the exponential map. In this paper, we consider only the case in which
    \begin{equation}\label{eq:Gam-sym}
        \Exp(t,\xi_1) = \Exp(t,\xi_2), \qquad \text{with}\;\xi_1 \in \Gamma_\omega,\;\;\xi_2 \in \Gamma_{-\omega},\;\;\xi_1\neq \xi_2.
    \end{equation}
    For brevity, we refer to such Maxwell points as \textbf{Maxwell points corresponding to the $\Gamma_\omega$-symmetry}.\footnote{From the perspective of Hamiltonian dynamics, Lie–Poisson reduction restricts the geodesic flow to two-dimensional coadjoint orbits, given by the level sets of the Casimir function. The resulting one-degree-of-freedom Hamiltonian system is therefore integrable. From this viewpoint, geodesics corresponding to the $\Gamma_\omega$-symmetry have the same values of the Hamiltonian and the Casimir function.}
    By the $\SO$-symmetry and the transitivity of its action on $\Gamma_\omega$, an identity analogous to Eq. \eqref{eq:con-Gamma} holds for these Maxwell points. Hence, the \textbf{first Maxwell time} corresponding to the $\Gamma_\omega$-symmetry is the same for every geodesic with initial velocity in $\Gamma_\omega$. We denote this time by $\Tmax(\omega)$. By convention, $\Tmax(\omega) = \infty$ if no Maxwell points corresponding to the \(\Gamma_\omega\)-symmetry occur along \(\gamma\). 

    \begin{Theorem}\label{thm:Max-time}
        The first Maxwell time $\Tmax(\omega)$ corresponding to the $\Gamma_\omega$-symmetry is the smallest positive solution $t$ to the following  transcendental equations:
        \begin{itemize}
            \item For $\omega \in [-1,-\frac{1}{\sqrt{2}}) \cup (\frac{1}{\sqrt{2}},1]$,
             \begin{equation*}
            \cos(t\sqrt{2\omega^2-1}) \sin(2\omega t)  -\frac{\omega}{\sqrt{2\omega^2-1}} \sin(t\sqrt{2\omega^2-1})\cos(2\omega t) = 0.
        \end{equation*}
        
        \item For $\omega = \pm\frac{1}{\sqrt{2}}$,
        \begin{equation*}
            \sin(2\omega t) -  \omega t \cos(2\omega t) = 0.
        \end{equation*}
        \item For $\omega \in (-\frac{1}{\sqrt{2}},0)\cup (0,\frac{1}{\sqrt{2}})$,
        \begin{equation*}
    \begin{split}
        \cosh(t\sqrt{1-2\omega^2}) \sin(2\omega t)  -\frac{\omega}{\sqrt{1-2\omega^2}} \sinh(t\sqrt{1-2\omega^2})\cos(2\omega t) = 0.
    \end{split}
    \end{equation*}
        \end{itemize}
       Furthermore, for $\omega\neq 0$, $\Tmax(\omega) \in (\frac{\pi}{2|\omega|},\frac{\pi}{|\omega|}]$. 
    \end{Theorem}
    Two important observations follow from Theorem \ref{thm:Max-time}. First, the construction in Eq. \eqref{eq:Gam-sym} is invariant under $\omega \to -\omega$, and hence $\Tmax(\omega) = \Tmax(-\omega)$. Second, since $\Tmax(\omega) \in (\frac{\pi}{2|\omega|},\frac{\pi}{|\omega|}]$, we have $\Tmax(\omega) \to \infty$ as $\omega \to 0$. Accordingly, we set $\Tmax(0) = \infty$.

    A direct consequence of Theorems \ref{thm:classic} and \ref{thm:Max-time} is the following corollary.
    \begin{Cor}\label{cor:t-max-t-con}
    For $\omega \in (-1,0)\cup (0,1)$, $ \Tmax(\omega) <  \Tcon (\omega)$, while $\Tmax(0) =  \Tcon (0) = \infty$ and $ \Tmax(\pm1) =  \Tcon (\pm1) = \pi$.
    \end{Cor}

    \begin{Prop}\label{prp:f-max-t-reg}
    The restriction $\Tmax\vcentcolon(0,1)\to (0,\infty)$ is $C^1$ and attains its unique global minimum at
    $$\omega = \frac{3}{\sqrt{14}} ,\qquad \Tmax\bigg(\frac{3}{\sqrt{14}}\bigg) = \frac{\pi\sqrt{14}}{4}. $$
    \end{Prop}

    To study the injectivity of the exponential map as an optimal control problem, recall that sufficiently short geodesic segments are length minimizing. The \textbf{cut time} is the time at which a geodesic ceases to be minimizing. More precisely, for $g\in \SLg$ and $v\in T_g\SLg$, consider the geodesic $\gamma_{g,v}$ satisfying the initial conditions
    $$ \gamma_{g,v}(0) = g, \qquad \dot{\gamma}_{g,v}(0) = v.$$
    The cut time of $\gamma_{g,v}$ is defined by
    $$ \Tcut (g,v)\coloneqq \sup\{ t> 0\vcentcolon \gamma_{g,v}|_{[0,t]} \;\text{is minimizing}\}.$$
    The point $\gamma_{g,v}(\Tcut (g,v))$ is called the \textbf{cut point} along the geodesic $\gamma_{g,v}$.
    By convention, $\Tcut (g,v) =\infty$ if $\gamma_{g,v}$ has no cut points. The geodesic $\gamma_{g,v}$  is called \textbf{globally minimizing} if $\Tcut (g,v)=\Tcut (g,-v)= \infty$.

    Since we consider only geodesics starting at the identity, we write $\Tcut (\xi)\coloneqq \Tcut (\Id,\xi)$. The third fundamental set for our study is the \textbf{cut locus}, consisting of all cut points:
     $$ \mathcal{Cut}(\Id)\coloneqq \{ \Exp(\Tcut (\xi),\xi)\vcentcolon \xi\in \mathbb{S}^2(\sla),\;\; \text{and} \;\; \Tcut (\xi) <\infty\}. $$

     By $\SO$-symmetry and the transitivity of its action on $\Gamma_\omega$, the cut time is the same for every geodesic with initial velocity in $\Gamma_{\omega}$. We denote this time by $\Tcut (\omega)$.

     Maxwell points are important because a geodesic cannot remain minimizing beyond its first Maxwell point. In particular, the first Maxwell time corresponding to the $\Gamma_\omega$-symmetry provides an upper bound for the cut time:
     $$\Tcut (\omega) \leq \Tmax(\omega).$$
     In general, proving that the first Maxwell time coincides with the cut time is the most delicate part of determining the optimal synthesis, as it requires showing that the Maxwell points obtained from the chosen symmetry determine the cut locus. Establishing this fact typically requires a topological analysis of the image of the exponential map after removing the Maxwell stratum. We now make this approach precise.

      Since the construction in Eq. \eqref{eq:Gam-sym} is invariant under $\omega \to -\omega$, we need only consider \(\omega>0\). We therefore define the \textbf{Maxwell stratum}\footnote{ The terminology \enquote{Maxwell stratum} originates in singularity theory and was introduced by V. I. Arnold in \cite{arnold2013singularities}.} by 
    $$\MAS\coloneqq \bigcup_{\omega \in (0,1]} \Big\{\Exp(\Tmax(\omega),\xi)\vcentcolon\xi\in\Gamma_\omega\Big\}  . $$
    
    \begin{lemma}\label{lem:g-ome-sym}
    The Maxwell stratum is
    $$ \MAS = \{ g\in\SLg\vcentcolon g=g^\T,\quad g<0\};$$
     that is, $\MAS$ consists precisely of the symmetric negative-definite matrices in $\SLg$.\footnote{Under the identification of \(\mathrm{SL}_2(\mathbb R)\) with \(\mathrm{SU}(1,1)\) used by Podobryaev and Sachkov \cite{PodobryaevSachkov2018}, their coordinate \(q_3\) is proportional to the difference between the off-diagonal entries. Thus, \(q_3=0\) corresponds to the symmetric locus $g=g^T$, while the Maxwell stratum consists of its negative-definite matrices.}
    \end{lemma}
    Motivated by the first Maxwell time, we define
    \begin{equation}\label{eq:max-stratum}
    \begin{split} 
        \MD & \coloneqq \{0_{2}\}\cup \Big(\bigcup_{\omega \in [-1,1]}  \Psi\big((0,\Tmax(\omega))\times \Gamma_\omega\big) \Big) \subseteq \sla ,\\
        \ND & \coloneqq \SLg \setminus \MAS.
    \end{split}
    \end{equation}
     Our third main result establishes that the exponential map is a global diffeomorphism when restricted to the domain determined by the first Maxwell time.

    \begin{Theorem}\label{thm:main-3}
    The restriction 
    $$\Exp\vcentcolon\MD \to \ND$$
    is a diffeomorphism. 
    \end{Theorem}

    Theorem \ref{thm:main-3}, together with the preceding results, yields the following consequences for the global optimality of the geodesics.
    \begin{Cor}\label{cor:cut-time}
    For $\omega \in[-1,1]$,
    $$ \Tcut (\omega) = \Tmax(\omega). $$
    Therefore,
    $$ \mathcal{Cut}(\Id) = \MAS.$$
    Moreover, every geodesic with initial velocity in $\Gamma_0$ is globally minimizing.
    \end{Cor}

    \begin{Cor}\label{cor:rad-in}
    The injectivity radius of $\SLg$ is
    $$\Tmax\bigg(\frac{3}{\sqrt{14}}\bigg) = \frac{\pi\sqrt{14}}{4}.$$
    \end{Cor}

    \begin{Cor}\label{cor:cut-luc-inter}
    The cut locus and the conjugate locus of $\Id$ intersect only at $-\Id$; that is, 
        $$ \mathcal{Cut}(\Id) \cap \mathcal{Con}(\Id) = \{ -\Id\}.$$
    \end{Cor}

    Our characterization of the cut time also yields an equi-optimality property of the geodesics, which we now describe. A geodesic $\gamma\vcentcolon\R \to \SLg$ is \textbf{equi-optimal} if its cut time is invariant under time translation. That is, $\Tcut (\gamma(t),\dot{\gamma}(t))$ is independent of $t\in \R$. 

    \begin{Cor}\label{cor:equi-opt}
    Every geodesic on $\SLg$ is equi-optimal.
    \end{Cor}

    Our approach to studying the injectivity of the exponential map is based on optimal synthesis and proceeds through the following steps:
    \begin{enumerate}
        \item Integrate the geodesic equations.
        \item Compute the conjugate times.
        \item Study the symmetries of the exponential map.
        \item Determine the first Maxwell time and prove that the first conjugate time is greater than or equal to the first Maxwell time.
        \item Apply Hadamard's global diffeomorphism theorem to the restriction $\Exp\vcentcolon\MD \to \ND$.
    \end{enumerate}
    
    Our approach treats the exponential map as an optimal control problem and uses optimal synthesis to prove Theorem \ref{thm:main-3}. The paper is organized as follows. In Section \ref{sec:loc}, we integrate the geodesic equations using the Euler–Poincar\'e equation and compute the first conjugate time and the conjugate locus. In Section \ref{sec:glo}, we analyze the symmetries of the exponential map, determine the first Maxwell time, and compare it with the first conjugate time. 
    Finally, in Section \ref{sec:top}, we study the topology of the Maxwell stratum and apply Hadamard’s global diffeomorphism theorem to establish Theorem \ref{thm:main-3}. This allows us to characterize the cut time and cut locus and compute the injectivity radius.

 \section{ Local Study}\label{sec:loc}

 In this section, we study the geodesic flow on $\SLg$ and determine the conjugate points along its geodesics. In Subsection \ref{subsec:SL2}, we define $\SLg$ and describe its left-invariant structure. In Subsection \ref{subsubsec:eul-poin}, we present the general theory of the Euler–Poincar\'e equation and its linearization. We then specialize these equations to $\SLg$ and obtain explicit formulas for the geodesics and Jacobi fields. In Subsection \ref{subsec:con-poin}, we use these formulas to determine the conjugate time. We conclude by characterizing the conjugate locus and proving Theorem \ref{thm:classic}.

\subsection{The Special Linear Group}\label{subsec:SL2}

   The special linear group $\SLg$ is defined as the Lie subgroup of the general linear group $\mathrm{GL}_2(\mathbb{R})$ given by
\[\SLg \coloneqq\{g\in \mathrm{GL}_2(\mathbb{R}) \;\vcentcolon\;\det(g)=1\}.\]
An important subgroup of $\SLg$ is the special orthogonal group $\mathrm{SO}(2)$, consisting of rotations. This subgroup is defined by
\[  \mathrm{SO}(2) \coloneqq \{ g \in \SLg \vcentcolon g^\T g = \Id\}. \]

We denote the Lie algebra of $\SLg$ by $\sla$. Since $\det(\exp(t\xi)) = e^{t\Tr(\xi)}$, the curve $\exp(t\xi)$ belongs to $\SLg$ for every $t\in \R$ if and only if $\Tr(\xi) =0$. Hence
$$ \sla \coloneqq \{ \xi \in M_2(\mathbb{R}) \vcentcolon  \Tr(\xi) = 0 \}.  $$
The Lie bracket is the matrix commutator 
   $$[\xi,\eta] = \xi \eta-\eta \xi, \qquad \xi,\eta \in \sla.$$
For further details on the preceding discussion, see \cite[Chapter 5]{stillwell-naive-lie}.

As mentioned earlier, the adjoint action of $\SO$ preserves the inner product on $\sla$. Indeed, using the cyclic property of the trace, one obtains
\begin{equation*}
    \langle \xi,\eta \rangle = \langle \Ad_Q(\xi) ,\Ad_Q(\eta) \rangle,\;\text{ for all }  \xi,\eta \in \sla \text{ and } Q \in  \mathrm{SO}(2). 
\end{equation*}
Since the metric is left-invariant and the differential of  $\Phi_Q$ at the identity is $\Ad_Q$, it follows that the conjugation map
$$\Phi_Q\vcentcolon \SLg \to \SLg, \qquad \Phi_Q(A) = Q AQ^{-1}$$
s an isometry for every  $Q\in \mathrm{SO}(2)$. As explained earlier, we investigate the minimality properties of geodesics in $\SLg$ by considering a representative from each orbit of the $\SO$-action. To describe this strategy in more detail, we first introduce the following orthonormal basis for $\sla$:
\begin{equation}\label{eq:Lie-alg-basis}
\xi_1=\left[\begin{array}{cc}
    1 & 0 \\
    0 & -1
\end{array}\right], \qquad 
\xi_2=\left[\begin{array}{cc}
    0 & 1 \\
    1 & 0
\end{array}\right]
 , \qquad
\xi_3=\left[\begin{array}{cc}
    0 & -1 \\
    1 & 0
\end{array}\right]
.    
\end{equation}
The vectors $\xi_1$ and $\xi_2$ are symmetric matrices, whereas $\xi_3$ is skew-symmetric. In addition, $\xi_3$ spans the Lie algebra of the subgroup $\mathrm{SO}(2)$; see \cite[Chapter 5]{stillwell-naive-lie} for more details. 

The elements of the basis satisfy the following Lie bracket relations:
\begin{equation}\label{eq:lie-al-rel}
 [\xi_1,\xi_2]=-2\xi_3 , \qquad [\xi_1,\xi_3]= -2\xi_2  , \qquad [\xi_2,\xi_3]=2\xi_1 .     
\end{equation}

Now, recall the definition of $\Gamma_\omega$ in terms of this basis:
\begin{equation*}
    \Gamma_\omega  \coloneqq \left\{ \xi \in \mathbb{S}^2(\sla) \vcentcolon \omega\xi_3 = \frac{1}{2}(\xi-\xi^\T)  \right\}.
\end{equation*}
The following proposition states that the action of $\mathrm{SO}(2)$ on $\Gamma_\omega$ is transitive. 
\begin{Prop}\label{prp:iso-action}
    For $\omega\in [-1,1]$, the set \(\Gamma_\omega\) is invariant under the adjoint action of $\SO$. Moreover, this action is transitive; that is, for aby $\xi,\eta \in \Gamma_\omega$, there exists $Q\in \mathrm{SO}(2)$ such that $\Ad_{Q}(\xi) = \eta$.   
\end{Prop}
\begin{proof}
  The  Lie bracket relations in Eq. \eqref{eq:lie-al-rel}, together with the identity
  \begin{equation}\label{eq:exp-equi}
       \Ad_{\exp(\xi)} = e^{\ad_\xi},
   \end{equation}
  imply that
  \begin{equation*}
      \begin{split}
         \Ad_{\exp(-\theta \xi_3)}(\xi_1) & = \cos(2\theta)\xi_1-\sin(2\theta)\xi_2, \\ 
         \Ad_{\exp(-\theta \xi_3)}(\xi_2) & = \sin(2\theta)\xi_1 + \cos(2\theta)\xi_2, \\
         \Ad_{\exp(-\theta \xi_3)}(\xi_3) & = \xi_3 .
      \end{split}
  \end{equation*}
  These equations show that $\Ad_{\exp(-\theta \xi_3)}$ acts as a rotation through angle $-2\theta$ on $\mathrm{span}\{\xi_1,\xi_2\}$ and fixes $\xi_3$. Consequently, $\Gamma_\omega$ is invariant under the adjoint action of $\SO$. Since every $\xi \in \Gamma_\omega$ has $\xi_3$-component equal to $\omega$ and symmetric component of norm $\sqrt{1-\omega^2}$, and since $\Ad_{\exp(-\theta \xi_3)}$ realizes every rotation of  $\mathrm{span}\{\xi_1,\xi_2\}$, the action is transitively on $\Gamma_{\omega}$. 
\end{proof}

The following remark will play a fundamental role in the next section. 
\begin{Remark}\label{rem:Gam-par}
    Geometrically, $\Gamma_\omega$ is a circle of latitude on the sphere $\mathbb{S}^2(\sla)$. More precisely, for $|\omega|<1$, $\Gamma_\omega$ is a circle of radius $\sqrt{1-\omega^2}$ and height $\omega$; for $\omega = \pm 1,$ it reduces to a single point. Thus, we parametrize $\Gamma_\omega$ by an angle $\phi \in [0,2\pi)$ as follows:
    $$ \Gamma_\omega = \{ \xi = \sqrt{1-\omega^2}(\cos\phi \xi_1+\sin\phi \xi_2) + \omega \xi_3 \vcentcolon \phi\in [0,2\pi) \}. $$
    Under this parametrization, the determinant of $\xi$ is given by
    $$ \det(\xi) = 2\omega^2-1. $$
    Therefore, the sign of $\det(\xi)$ is  characterized as follows:
    \begin{equation*}
        \begin{split}
            \det(\xi) > 0 & \iff  \omega \in \Big[-1,-\frac{1}{\sqrt{2}}\Big)\cup \Big(\frac{1}{\sqrt{2}},1\Big],\\
            \det(\xi) = 0 &  \iff   \omega = \pm \frac{1}{\sqrt{2}}, \\
            \det(\xi) < 0 &  \iff \omega \in \Big(-\frac{1}{\sqrt{2}},\frac{1}{\sqrt{2}}\Big).\\
        \end{split}
    \end{equation*}
\end{Remark}

\subsubsection{The Euler-Poincar\'e Equation}\label{subsubsec:eul-poin}
In this subsection, we briefly introduce the Euler-Poincar\'e equation; see \cite[Appendix 2]{arnold1989mathematical} for further details. Consider a Lie group $\mathbb{G}$ and a smooth curve $\gamma\vcentcolon\R \to \mathbb{G}$.  The velocity vector $\dot{\gamma}(t)$ lies in the tangent space $T_{\gamma(t)}\mathbb{G}$. To compare velocities at different points of the group, we use the differential of left translation by $\gamma^{-1}(t)$ to transport $\dot{\gamma}(t)$ to the identity. The resulting Lie-algebra-valued curve
$$ u(t) = (dL_{\gamma^{-1}(t)})_{\gamma(t)} \dot{\gamma}(t) \in \mathfrak{g}$$
is called the \textbf{left-trivialized velocity}. For a matrix Lie group, this expression reduces to
$$ u(t) = \gamma^{-1}(t) \dot{\gamma}(t) \in \mathfrak{g}.$$
Let $\mathbb{G}$ be a Lie group equipped with the left-invariant Riemannian metric. The corresponding reduced kinetic energy Lagrangian $\ell\vcentcolon\mathfrak{g}\to \R$ is given by 
$$ \ell(\xi) \coloneqq \frac{1}{2} \langle \xi,\xi \rangle. $$
The geodesic equation is equivalent to the Euler-Poincar\'e equation
$$ \dot{u}(t) = \ad^*_u{\frac{\partial \ell}{\partial u}},$$
where $\frac{\partial \ell}{\partial u} \in \Lag^*$ denotes the differential of the reduced Lagrangian, and $\ad^*$ is defined by 
$$ \langle\ad^*_um,v\rangle = \langle m,[u,v]\rangle, \qquad m\in\Lag^*,\;\;v\in\Lag.$$
Thus, as mentioned earlier, the Euler–-Poincar\'e equation reduces the study of geodesics on the Lie group to an evolution equation on its Lie algebra. Now, recall that to recover the geodesic $\gamma(t)$ from a solution to the Euler-Poincar\'e equation, we set up the initial value problem
$$\dot{\gamma}(t) = (dL_{\gamma^{-1}(t)})_{\gamma(t)}u(t), \quad \gamma(0) = \Id. $$
For matrix Lie groups, this becomes
$$ \dot{\gamma}(t) = \gamma(t) u(t), \quad \gamma(0) = \Id. $$

This formulation also allows us to interpret the geodesic problem as an optimal control problem. Given $g_0,g_1 \in \mathbb{G}$, consider the control system
$$ \dot{\gamma}(t) = \gamma(t) u(t),\qquad u(t) \in \Lag,$$
with $\gamma(0) =g_0$ and $\gamma(T) = g_1$, and the energy functional 
$$J(u) = \frac{1}{2}\int_0^T \langle u(t),u(t)\rangle dt .$$
For fozed $T>0$, the trajectories that minimize \(J\) are constant-speed minimizing geodesics of the left-invariant Riemannian metric \cite{agrachev2013control}. Thus, the Euler-Poincar\'e equation determines the evolution of the reduced velocity \(u(t)\), while the reconstruction equation determines the corresponding geodesic on $\mathbb{G}$.

In our case, we write
\begin{equation*}
    u(t) = u_1(t)\xi_1+u_2(t)\xi_2+u_3(t)\xi_3,
\end{equation*}
then the reduced  kinetic energy Lagrangian $\ell\vcentcolon\sla \to \R$  has the form
$$\ell(u) = \frac{1}{2}\big(u_1^2+u_2^2+u_3^2\big).$$
The system of ODEs governing the reduced geodesic equations on $\sla$ is 
\begin{equation}
    \begin{split}
        \dot{u}_1(t) & = 4 u_3u_2,\\
        \dot{u}_2(t) & = -4 u_3u_1,\\
        \dot{u}_3(t) & = 0.\\
    \end{split}
\end{equation}
Thus, by imposing the energy condition $\ell(u) = \frac{1}{2}$, which implies that the geodesics are parametrized by arc length, we obtain the general solution
\begin{equation*}
    u(t) = \sqrt{1-\omega^2}\big( \cos(4\omega t+ \phi_0) \xi_1 - \sin(4\omega t+ \phi_0) \xi_2 \big) + \omega \xi_3,\;\; \omega \in [-1,1], \;\; \phi_0\in [0,2\pi). 
\end{equation*}
 Thus, the trajectory of $u(t)$ lies on the circle of latitude $\Gamma_\omega$ defined in Remark \ref{rem:Gam-par}. The reduced dynamics are stationary for $\omega\in\{\pm 1,0\}$ and describe a uniform circular motion on $\omega \in (-1,0)\cup (0,1)$. Moreover, the adjoint action of $\mathrm{SO}(2)$ preserves the reduced kinetic energy and maps solutions of the Euler-Poincar\'e equation to solutions. In particular, 
$$ \Ad_{\exp(-\theta\xi_3)}(u(t)) = \sqrt{1-\omega^2}\big( \cos(4\omega t+ \phi_0+2\theta) \xi_1 - \sin(4\omega t+ \phi_0+2\theta) \xi_2 \big) + \omega \xi_3. $$

To simplify our proofs, we choose the element
\begin{equation}\label{eq:xi-ome}
    \xi_\omega \coloneqq \sqrt{1-\omega^2} \xi_1 + \omega \xi_3,
\end{equation}
as a representative of $\Gamma_\omega$, and we denote by $u_\omega(t)$ the solution of the Euler-Poincar\'e equation with initial condition $u_\omega(0) = \xi_\omega$; that is
\begin{equation}\label{eq:geo-eq-sl2}
    u_\omega(t) = \sqrt{1-\omega^2}\big( \cos(4\omega t) \xi_1 - \sin(4\omega t) \xi_2 \big) + \omega \xi_3.
\end{equation}
Thus, the strategy is to prove our main results for geodesics $\gamma_\omega(t)$ with initial conditions
\begin{equation}\label{eq:def-gam-ome}
    \gamma_\omega(0) = \Id,\quad\text{and}\quad \dot{\gamma}_\omega(0) = \xi_\omega.
\end{equation}
 We then use the $\SO$-isometries to extend these optimality properties to all geodesics with initial velocities in $\Gamma_\omega$.

With this discussion in mind, we now present the explicit integration of the above equation, which will play an essential role in the rest of this work.

\begin{lemma}\label{lem:expl-form}
    Let $\xi \in \sla$. Then the geodesic $\gamma(t)$ with the initial conditions $\gamma(0) = \Id$ and $\dot{\gamma}(0) = \xi$ has the explicit form
    \begin{equation*}
        \gamma(t) = \exp(t\xi^\T)\exp(t(\xi-\xi^\T)). 
    \end{equation*}
    Furthermore, if $\xi \in \mathbb{S}^2(\sla)$, using the notation introduced in Remark \ref{rem:Gam-par}, then the following statements hold:
    \begin{itemize}
        \item If  $\;\det (\xi)>0$, then 
        \begin{equation*}
        \gamma(t) = \left(\cos(t\sqrt{2\omega^2-1} ) \Id + \frac{\sin(t \sqrt{2\omega^2-1})}{\sqrt{2\omega^2-1}} \xi^\T \right)\exp(t(\xi-\xi^\T)). 
    \end{equation*}
        \item If  $\;\det (\xi) = 0$, then 
        \begin{equation*}
        \gamma(t) = ( \Id + t \xi^\T)\exp(t(\xi-\xi^\T)). 
        \end{equation*}
         \item If  $\;\det (\xi)<0$, then 
        \begin{equation*}
        \gamma(t) = \left(\cosh(t \sqrt{1-2\omega^2}) \Id + \frac{\sinh(t\sqrt{1-2\omega^2})}{\sqrt{1-2\omega^2}} \xi^\T \right)\exp(t(\xi-\xi^\T)). 
        \end{equation*}
    \end{itemize}
\end{lemma}

\begin{proof}
     Let $\gamma\vcentcolon\R \to \SLg$ be the curve given by the explicit form. By $\SO$-symmetry and the homogeneity of the geodesic equation, it suffices to verify the geodesic equation for the representative $\gamma_\omega(t)$. Indeed, isometries map geodesics to geodesics, and rescaling the initial velocity corresponds to an affine reparametrization. We therefore verify that the left-trivialized velocity of
     $\gamma_\omega(t)$ satisfies Eq. \eqref{eq:geo-eq-sl2}. Then, we use the product rule to compute $\dot{\gamma}_\omega(t)$ and find that
     \begin{align*}
        \dot{\gamma}_\omega(t)&=\gamma_\omega(t)(\Ad_{\exp(-t(\xi_\omega-\xi_\omega)^\T)}(\xi_\omega)^\T)+\xi_\omega-\xi_\omega^\T)\\
        &=\gamma_\omega(t)(\sqrt{1-\omega^2 }\Ad_{\exp(-2t\omega \xi_3)}(\xi_1)+\omega \xi_3)\\
        &=\gamma_\omega(t)\big(\sqrt{1-\omega^2 }\big( \cos(4\omega t) \xi_1 -\sin(4\omega t) \xi_2 \big)+\omega \xi_3\big). 
    \end{align*}
    So $\gamma^{-1}_\omega(t)\dot{\gamma}_\omega(t) = u_\omega(t)$, where $u_\omega(t)$ is the solution of the Euler-Poincar\'e equation given by Eq. \eqref{eq:geo-eq-sl2}. In particular, $u_\omega(0) =\xi_\omega$ as required.

    By the Cayley–Hamilton theorem, every $\xi \in \sla$ satisfies $\xi^2 = -\det(\xi) \Id$. It follows that
    \begin{align*}
        \exp(t\xi^\T) &= \sum_{i=0}^\infty \frac{(-\det(\xi))^i t^{2i}}{(2i)!} \Id + \sum_{i=0}^\infty \frac{(-\det(\xi))^i t^{2i+1}}{(2i+1)!} \xi^\T. 
    \end{align*}
    Therefore, Remark \ref{rem:Gam-par} and the Taylor series for the trigonometric and hyperbolic functions yield the desired result. 
\end{proof}

\begin{Remark}\label{rem:t-rev-prp}
    Since the Riemannian geodesic equation is invariant under time reversal, if $\gamma(t)$ is a geodesic, then $\gamma(-t)$ is also a geodesic. Let $\gamma_\omega^-$ denote the geodesic with initial conditions
    $$ \gamma_\omega^-(0) = \Id, \quad\text{and} \quad \dot{\gamma}^-_\omega(0) = - \xi_\omega. $$
    By time reversibility and uniqueness of geodesics,
    $$ \gamma_\omega(-t) = \gamma_\omega^-(t).$$
    Consequently, by the $\SO$-symmetry and time reversibility, it suffices to study the optimal properties of $\gamma_\omega$ for $\omega\in [0,1]$.  
\end{Remark}

\subsubsection{The Linearized Euler-Poincar\'e Equation}

The Jacobi equation for Lie groups with left-invariant metrics was studied by Herring and O'Sullivan \cite{herring1979jacobi}; see also \cite{LeBrigantLichtenfelzPreston} for a modern treatment.

The classical method for finding Jacobi fields is to solve the second-order ODE known as the Jacobi equation; see \cite{do2016differential} for more details.  For a left-invariant Riemannian metric on a Lie group, the Jacobi equation can be expressed as a coupled system of equations on the Lie algebra, known as the linearized Euler-Poincar\'e equation system. This system describes the variations of both the reconstructed curve and its left-trivialized velocity. Thus, on a matrix Lie group, a Jacobi field along $\gamma(t)$ has the form
$$ W(t) = \gamma(t)\eta(t),  $$
where $(\eta(t),\zeta(t))$ satisfy the system of ODEs
$$ \dot{\eta} (t) =- \ad_{u(t)}(\eta(t)) + \zeta(t),\qquad\dot{\zeta}(t) = \ad_{u(t)}^\ast(\zeta(t)) + \ad_{\zeta(t)}^\ast (u(t)).$$
Here, we identify $\Lag$ with its dual $\Lag^*$ through the inner product induced by the Riemannian metric.

In our case, substituting the expression for $u_\omega(t)$ from Eq. \eqref{eq:geo-eq-sl2} into the linearized Euler-Poincar\'e equation yields the following coupled system:
\begin{equation}\label{eq:lin-eul-poi-1}
    \begin{split}
       \dot{\eta}_1(t) &= \zeta_1 + 2 \omega\eta_2+ 2\sqrt{1-\omega^2}\sin(4\omega t)\eta_3,\\
       \dot{\eta}_2(t) & = \zeta_2 -2 \omega\eta_1 + 2\sqrt{1-\omega^2}\cos(4\omega t)\eta_3 ,\\
      \dot{\eta}_3(t) & = \zeta_3 + 2\sqrt{1-\omega^2}\big(\cos(4\omega t)\eta_2+\sin(4\omega t)\eta_1\big), 
    \end{split}
\end{equation}
        
and 
\begin{equation}\label{eq:lin-eul-poi-2}
    \begin{split}
       \dot{\zeta}_1(t)&=4\omega \zeta_2 - 4\sqrt{1-\omega^2}\sin(4 \omega t)\zeta_3,  \\
       \dot{\zeta}_2(t) &=- 4\omega \zeta_1 - 4\sqrt{1-\omega^2}\cos(4\omega t)\zeta_3, \\
       \dot{\zeta}_3(t) &=0. 
    \end{split}
\end{equation}

The space of Jacobi fields vanishing at $t=0$ is three-dimensional. One such field arises from variations of the initial speed and is given by
$$ W_1(t) = t \gamma_\omega(t)u_\omega(t). $$
Since $W_1(t)$ vanishes only at $t=0$, conjugate points at positive times occur precisely when the three Jacobi fields $W_1(t)$, $W_2(t)$ and $W_3(t)$ become linearly dependent. The next section finds the Jacobi fields $W_2(t)$ and $W_3(t)$. By Remark \ref{rem:t-rev-prp} and $\SO$-symmetry, it suffices to consider $\omega \in [0,1]$. We distinguish five cases:
\textbf{(1)} $\omega = 1$; \textbf{(2)} $\omega \in (\frac{1}{\sqrt{2}},1)$; \textbf{(3)}  $\omega =\frac{1}{\sqrt{2}}$; \textbf{(4)}  $\omega \in (0,\frac{1}{\sqrt{2}})$; \textbf{(5)} $\omega  = 0$. 


\subsection{Conjugate Points}\label{subsec:con-poin}

The strategy for solving the linearized Euler-Poincar\'e equation is as follows: We use the fact that the infinitesimal generators of the left action of a Lie group $\mathbb{G}$ on itself are the right-invariant vector fields. Since the metric is left-invariant, these vector fields generate one-parameter families of isometries and are therefore Killing fields. Hence, their restrictions to the geodesic $\gamma_\omega(t)$ are Jacobi fields. Next, we use the adjoint action to express the right-invariant vector field $X^R_\xi(t)$, with $\xi\in \mathfrak{g}$, in the left-trivialization form as
$$ X^R_\xi\big(\gamma_\omega(t)\big) = \gamma_\omega(t)(\Ad_{\gamma_\omega^{-1}(t)} \xi). $$
Consequently, the left trivializations of three linearly independent right-invariant vector fields form a fundamental set of solutions of the homogeneous part of Eq. \eqref{eq:lin-eul-poi-1}. We first solve the linear system in Eq. \eqref{eq:lin-eul-poi-2}. Substituting its solutions into Eq. \eqref{eq:lin-eul-poi-1}, we then apply variation of parameters to obtain the general solution.

We omit the details of these computations and present the Jacobi fields vanishing at $t=0$, which are relevant to our analysis.

\subsubsection{Case $\omega = 1$} Consider $\omega = 1$. In this case, the geodesic $\gamma_\omega$ has the form
$$ \gamma_\omega(t) = \exp(t\xi_3),$$
and the linearized Euler-Poincar\'e equation is 
\begin{align*}
       \dot{\eta}_1(t) &= \zeta_1 +2 \eta_2  ,  \quad  & \dot{\zeta}_1(t)&=4 \zeta_2,\\
       \dot{\eta}_2(t) & = \zeta_2 - 2 \eta_1 , \quad &\dot{\zeta}_2(t) &=-4  \zeta_1  ,\\
       \dot{\eta}_3(t) & = \zeta_3 , \quad & \dot{\zeta}_3(t) &=0. 
    \end{align*}
Then, for $\omega = 1$, the basis of Jacobi fields vanishing at $t=0$ is completed by
\begin{equation}\label{eq:jac-vec-1}
    \begin{split}
        W_2(t) & = \sin(t) \gamma_\omega(t)\big(\sin(3t)\xi_1+ \cos(3t) \xi_2\big), \\
        W_3(t) & = \sin(t) \gamma_\omega(t)\big(\cos(3t)\xi_1- \sin(3t) \xi_2\big).
    \end{split}
\end{equation}

\begin{lemma}\label{lem:con-time-1}
    For $\omega = \pm 1$,  the conjugate times of $\gamma_\omega$ are $t = \pi k$ where $k$ is a positive integer. Each conjugate time has multiplicity 2. 
\end{lemma}
\begin{proof}
    If $\omega =1$, then the vector fields in Eq. \eqref{eq:jac-vec-1} are Jacobi fields along $\gamma_\omega$.  Since the vectors
    $$ \sin(3t)\xi_1+ \cos(3t) \xi_2\quad\text{and}\quad \cos(3t)\xi_1- \sin(3t) \xi_2 $$
    are linearly independent and span the $(\xi_1,\xi_2)$-plane, the fields $W_1, W_2, W_3$ are linearly independent whenever $\sin(t) \neq 0$. At $t=\pi k$, both $W_2$ and $W_3$ vanish, while $W_1$ does not. Hence, the conjugate times are $t= \pi k$, where $k$ is a positive integer; each conjugate time has multiplicity $2$. The case $\omega = -1$ follows from Remark \ref{rem:t-rev-prp} and $\SO$-symmetry. 
\end{proof}

\subsubsection{Case $\omega \in (\frac{1}{\sqrt{2}},1)$}

By Lemma \ref{lem:expl-form},  the corresponding geodesic is given by
$$\gamma_\omega(t) = \left(\cos(t \sqrt{2\omega^2-1}) \Id + \frac{\sin(t\sqrt{2\omega^2-1})}{\sqrt{2\omega^2-1}}\big(\sqrt{1-\omega^2}\xi_1 -\omega \xi_3\big) \right) \exp(2\omega t \xi_3). $$
The linearized Euler-Poincar\'e equations are given by Eqs. \eqref{eq:lin-eul-poi-1} and \eqref{eq:lin-eul-poi-2}.  Setting $\alpha(\omega) = \sqrt{2\omega^2-1}$, a basis of the space of Jacobi fields vanishing at $t=0$ is completed by
\begin{equation}\label{eq:jac-vec-2}
    \begin{split}
        W_2(t) & = \sin\!\big(t \alpha(\omega)\big) \gamma_\omega(t) \big( f_1(t)\xi_1 +f_2(t)\xi_2+f_3(t)\xi_3 \big),   \\
        W_3(t) & = \gamma_\omega(t)\Big(g_1(t)\xi_1 + g_2(t)\xi_2 +\frac{1}{2} \omega  \sqrt{1-\omega ^2} \sin\!\big(2 t \sqrt{2\omega^2-1}\big)\xi_3\Big).
    \end{split}
\end{equation}
Here
\begin{equation*}
    \begin{split}
        f_1(t) &= \alpha(\omega) \sin (4\omega t) \cos\!\big(t \alpha(\omega)\big)-\omega \sin\!\big(t \alpha(\omega)\big) \cos (4\omega t), \\
        f_2(t) &= \omega \sin (4\omega t) \sin\!\big(t \alpha(\omega)\big)+\alpha(\omega) \cos (4\omega t) \cos\!\big(t \alpha(\omega)\big), \\
   f_3(t) &= \sqrt{1-\omega^2} \sin\!\big(t \alpha(\omega)\big), \\
   g_1(t) &= -\alpha(\omega) \big(\cos (4\omega t) \big(2 t (\omega^2-1)+\frac{\omega ^2
   \sin\!\big(2 t \alpha(\omega)\big)}{2 \alpha(\omega)}\big) \\
   &\mathrel{\phantom{=}}+\omega\sin (4\omega t)\sin^2\!\big(t
   \alpha(\omega)\big)\big), \\
   g_2(t) &= \alpha(\omega) \big(\sin (4\omega t) \big(2 t (\omega ^2-1)+\frac{\omega ^2
   \sin\!\big(2 t \alpha(\omega)\big)}{2 \alpha(\omega)}\big) \\
   &\mathrel{\phantom{=}}-\omega  \sin^2\!\big(t
   \alpha(\omega)\big) \cos(4\omega t)\big). 
    \end{split}
\end{equation*}

   \begin{lemma}\label{lem:con-time-2}
    For $\omega \in (-1,-\frac{1}{\sqrt{2}}) \cup (\frac{1}{\sqrt{2}},1)$, the first conjugate time of $\gamma_\omega$ is $t = \frac{\pi }{\sqrt{2\omega^2-1}}$ and it has multiplicity 1. 
\end{lemma}
\begin{proof}
    For $\omega \in (-1,-\frac{1}{\sqrt{2}}) \cup (\frac{1}{\sqrt{2}},1)$, the vector fields in Eq. \eqref{eq:jac-vec-2} are Jacobi fields along $\gamma_\omega$. Setting \(\alpha(\omega)\coloneqq\sqrt{2 \omega^2-1}\), a direct computation gives
    \begin{equation*}
    \begin{split}
        \det(W_1,W_2,W_3) &=\omega \alpha(\omega) t \sin( \alpha(\omega)t)\big( \sin( \alpha(\omega)t) - 2(1-\omega^2)\alpha(\omega)t\cos(\alpha(\omega)t) \big). \\
    \end{split}
    \end{equation*}
    Since $0<2(1-\omega^2) <1$ and $\tan x > x$ for $x\in (0,\frac{\pi}{2})$, the expression in parentheses is strictly positive for $t \in (0,\frac{\pi}{2\alpha(\omega)})$. Consequently, the determinant does not vanish for $t \in (0,\frac{\pi}{2\alpha(\omega)})$ and vanishes at $t =\frac{\pi}{2\alpha(\omega)}$. Hence the first conjugate time is at $t=\frac{\pi}{\sqrt{2 \omega^2-1}}$. At this time, $W_2$ vanishes, while $W_1$ and $W_3$ are linearly independent. Therefore, the first conjugate time has multiplicity $1$. The case $\omega \in (-1,-\frac{1}{\sqrt{2}})$ follows from Remark \ref{rem:t-rev-prp} and $\SO$-symmetry.
    \end{proof}


\subsubsection{Case $\omega=\pm \frac{1}{\sqrt{2}}$}

Consider $\omega = \frac{1}{\sqrt{2}}$. In this case, the geodesic $\gamma_\omega$ has the form
$$\gamma_\omega(t) = \left( \Id + \frac{t}{\sqrt{2}} (\xi_1-\xi_3) \right)\exp(\sqrt{2}t \xi_3), $$
and the linearized Euler-Poincar\'e equation is
\begin{align*}
       \dot{\eta}_1(t) &= \zeta_1 + \sqrt{2}\eta_2+ \sqrt{2}\sin(2\sqrt{2}t)\eta_3,\\
       \dot{\eta}_2(t) & = \zeta_2 - \sqrt{2}\eta_1 + \sqrt{2}\cos(2\sqrt{2}t)\eta_3 ,\\
      \dot{\eta}_3(t) & = \zeta_3 + \sqrt{2}\big(\cos(2\sqrt{2} t)\eta_2+\sin(2\sqrt{2} t)\eta_1\big) , 
\end{align*} 
and 
\begin{align*}
       \dot{\zeta}_1(t)&=2\sqrt{2}(\zeta_2 - \sin(2\sqrt{2}t)\zeta_3),  \\
       \dot{\zeta}_2(t) &=- 2\sqrt{2}(\zeta_1 +\cos(2\sqrt{2}t)\zeta_3), \\
       \dot{\zeta}_3(t) &=0. 
\end{align*}
A basis of the space of Jacobi fields vanishing at $t=0$ is completed by
\begin{equation}\label{eq:jac-vec-3}
    \begin{array}{l}
        W_2(t)  = \frac{t}{\sqrt{2}} \gamma_\omega(t)\big( f_1(t)\xi_1+f_2(t)\xi_2+t\xi_3\big), 
   \\[1mm]
         W_3(t)  = \frac{t}{6}\gamma_\omega(t)\big(g_1(t)\xi_1- g_2(t)\xi_2+2  \big(3-t^2\big)\xi_3\big). 
    \end{array}
\end{equation}
For this case, the auxiliary functions $f_1,f_2,g_1,g_2\vcentcolon\R \to \R$ are defined by 
\begin{equation*}
\begin{split}
    f_1(t)& = \sqrt{2} \sin (2 \sqrt{2} t)-t \cos (2 \sqrt{2} t),\\
   f_2(t) & = t
   \sin (2 \sqrt{2} t)+\sqrt{2} \cos(2 \sqrt{2}
   t),\\
   g_1(t)& = 2 t^2 \cos(2 \sqrt{2}t)-3 \sqrt{2} t \sin(2 \sqrt{2} t),\\
   g_2(t) & = 2 t^2 \sin(2 \sqrt{2} t)+3 \sqrt{2} t \cos (2 \sqrt{2} t).
\end{split}
\end{equation*}

\begin{lemma}\label{lem:con-time-3}
   For $\omega = \pm\frac{1}{\sqrt{2}}$, the geodesic $\gamma_\omega$ has no conjugate times. 
\end{lemma}
\begin{proof}
    For $\omega = \frac{1}{\sqrt{2}}$, the fields in Eq. \eqref{eq:jac-vec-3} are Jacobi fields along $\gamma_\omega$.  A direct computation shows that
    \begin{equation*}
    \begin{split}
       \det(W_1,W_2,W_3)  = \frac{t^3 (t^2+3)}{3 \sqrt{2}}.
    \end{split}
    \end{equation*}
    Since this determinant is nonzero for every $t \neq 0$, the Jacobi fields remain linearly independent at all positive times. Hence, $\gamma_\omega$ has no conjugate points. The case $\omega = -\frac{1}{\sqrt{2}}$ follows from Remark \ref{rem:t-rev-prp} and $\SO$-symmetry. 
    \end{proof}


\subsubsection{Case $\omega \in  (0,\frac{1}{\sqrt{2}})$}

By Lemma \ref{lem:expl-form},  the corresponding geodesic is given by
$$\gamma_\omega(t) = \Big(\cosh(t \sqrt{1-2\omega^2}) \Id + \frac{\sinh(t\sqrt{1-2\omega^2})}{\sqrt{1-2\omega^2}} \big( \sqrt{1-\omega^2}\xi_1 - \omega\xi_3) \Big)\exp(2\omega t \xi_3). $$
The linearized Euler-Poincar\'e equations are given by Eqs. \eqref{eq:lin-eul-poi-1} and \eqref{eq:lin-eul-poi-2}.  Setting $\alpha(\omega) = \sqrt{1-2\omega^2}$, a basis of the space of Jacobi fields vanishing at $t=0$ is completed by
\begin{equation}\label{eq:jac-vec-4}
    \begin{split}
        W_2(t) & = \sinh\!\big(t \alpha(\omega)\big) \gamma_\omega(t) \big( f_1(t)\xi_1 +f_2(t)\xi_2+f_3(t)\xi_3 \big),   \\
        W_3(t) & = \gamma_\omega(t)\big(g_1(t)\xi_1 + g_2(t)\xi_2 +\frac{1}{2} \omega  \sqrt{1-\omega ^2} \sinh\!\big(2 t \alpha(\omega)\big)\xi_3\big).
    \end{split}
\end{equation}
For this case, the auxiliary functions $f_1,f_2,g_1,g_2\vcentcolon\R \to \R$ are defined by 
\begin{equation*}
    \begin{split}
        f_1(t) &= \alpha(\omega) \sin (4\omega t) \cosh\!\big(t \alpha(\omega)\big)-\omega \sinh\!\big(t \alpha(\omega)\big) \cos (4\omega t), \\
        f_2(t) &= \omega \sin (4\omega t) \sinh\!\big(t \alpha(\omega)\big)+\alpha(\omega) \cos (4\omega t) \cosh\!\big(t \alpha(\omega)\big), \\
   f_3(t) &= \sqrt{1-\omega^2} \sinh\!\big(t \alpha(\omega)\big), \\
   g_1(t) &= -\alpha(\omega) \big(\cos (4\omega t) \big(2 t (\omega ^2-1)+\frac{\omega ^2
   \sinh\!\big(2 t \alpha(\omega)\big)}{2 \alpha(\omega)}\big) \\
   &\mathrel{\phantom{=}}-\omega\sin (4\omega t)\sinh^2\!\big(t
   \alpha(\omega)\big)\big), \\
   g_2(t) &= \alpha(\omega) \big(\sin (4\omega t) \big(2 t (\omega ^2-1)+\frac{\omega ^2
   \sinh\!\big(2 t \alpha(\omega)\big)}{2 \alpha(\omega)}\big) \\
   &\mathrel{\phantom{=}}+\omega  \sinh^2\!\big(t
   \alpha(\omega)\big) \cos(4\omega t)\big). 
    \end{split}
\end{equation*}

\begin{lemma}\label{lem:con-time-4}
    For $\omega \in (-\frac{1}{\sqrt{2}},0)\cup (0,\frac{1}{\sqrt{2}})$, the geodesic $\gamma_\omega$ has no conjugate times. 
\end{lemma}

\begin{proof}
    For $\omega \in (-\frac{1}{\sqrt{2}},0)\cup (0,\frac{1}{\sqrt{2}})$, the vector fields in Eq. \eqref{eq:jac-vec-4} are Jacobi fields along $\gamma_\omega$. Setting \(\alpha(\omega)\coloneqq\sqrt{1-2\omega^2}\), a direct computation gives that
    \begin{equation*}
    \begin{split}
        \det(W_1,W_2,W_3) & = \omega t \alpha(\omega)\sinh(\alpha t)\cosh(\alpha t)
\big(\tanh(\alpha t)-2(1-\omega^2)\alpha(\omega) t\big).
    \end{split}
    \end{equation*}
    Since $2(1-\omega^2)>1$ and $|\tanh(x)|<|x|$ for every $x\neq 0$, the expression in parentheses is nonzero for every $t\neq 0$. Consequently, the determinant above vanishes only at $t=0$, and $\gamma_\omega$ has no conjugate times.
    \end{proof}


    %



\subsubsection{Case $\omega = 0$}

In this case, the geodesic $\gamma_\omega(t)$ has the form
$$ \gamma_\omega(t)  = \cosh(t ) \Id + \sinh(t) \xi_1 ,$$
and the linearized Euler-Poincar\'e equation is 
\begin{align*}
       \dot{\eta}_1(t) &= \zeta_1   ,  \quad  & \dot{\zeta}_1(t)&=0,\\
       \dot{\eta}_2(t) & = \zeta_2 +2\eta_3 , \quad &\dot{\zeta}_2(t) &= -4\zeta_3,\\
       \dot{\eta}_3(t) & = \zeta_3 +2\eta_2, \quad & \dot{\zeta}_3(t) &=0. 
    \end{align*}
Then, the basis of Jacobi fields vanishing at $t=0$ is completed by
\begin{equation}
    \begin{split}\label{eq:jac-vec-5}
    W_2(t) & =  \gamma_\omega(t)\big( \sinh(2t) \xi_2+ (\cosh(2t)-1) \xi_3\big), \\
    W_3(t) & = \gamma_\omega(t)\big( (\cosh(2t)-1) \xi_2- \sinh(2t) \xi_3\big).
    \end{split}
\end{equation}
\begin{lemma}\label{lem:con-time-5}
    Consider $\omega = 0$. The geodesic $\gamma_\omega(t)$ defined in Eq. \eqref{eq:def-gam-ome} has no conjugate times. 
\end{lemma}
\begin{proof}
    If $\omega =0$, the vector fields in Eq. \eqref{eq:jac-vec-5} are Jacobi fields along $\gamma_\omega$. 
    $$\det(W_1,W_2,W_3) = -t\big(\sinh^2(2t) + (\cosh(2t)-1)^2\big)$$
    The expression in parentheses is nonzero for every $t\neq 0$. Therefore, the determinant above vanishes only at $t=0$, and $\gamma_\omega$ has no conjugate times.
\end{proof}


\begin{proof}[Proof of Theorem \ref{thm:classic}] We divide the proof into two cases: $\omega \in [-1,-\frac{1}{\sqrt{2}})\cup(\frac{1}{\sqrt{2}},1]$ and $\omega \in [-\frac{1}{\sqrt{2}},\frac{1}{\sqrt{2}}]$.

\textbf{(Case $\omega \in [-1,-\frac{1}{\sqrt{2}})\cup(\frac{1}{\sqrt{2}},1]$)} Lemmas \ref{lem:con-time-1} and \ref{lem:con-time-2} imply that the first conjugate time of $\gamma_\omega$ is $\Tcon(\omega) = \frac{\pi}{\sqrt{2\omega^2-1}}$. 
By $\SO$-symmetry and Proposition \ref{prp:iso-action}, this formula holds for every geodesic with initial velocity in $\Gamma_\omega$.

\textbf{(Case $\omega \in [-\frac{1}{\sqrt{2}},\frac{1}{\sqrt{2}}]$)} Lemmas \ref{lem:con-time-3}, \ref{lem:con-time-4} and \ref{lem:con-time-5} imply that $\gamma_\omega$ has no conjugate times. By the $\SO$-symmetry and Proposition \ref{prp:iso-action}, the same holds for every geodesic with initial velocity in $\Gamma_\omega$. Therefore, $\Tcon (\omega) =\infty$.   

By Eq. \eqref{eq:def-gam-ome}, we have that if $\omega \in [-1,-\frac{1}{\sqrt{2}})\cup(\frac{1}{\sqrt{2}},1]$, then 
$$ \Exp(\Tcon (\omega),\xi_\omega)  = -\exp\!\bigg(\frac{\pi}{\sqrt{2\omega^2-1}}(\xi_\omega-\xi_\omega^\T)\bigg) = -\exp\!\bigg(\frac{2\pi\omega}{\sqrt{2\omega^2-1}}\xi_3\bigg) \in \SO.$$
Since $\frac{2\pi\omega}{\sqrt{2\omega^2-1}} \to \infty$ as $\omega \to \frac{1}{\sqrt{2}}^+$, and this function is continuous, its image contains an interval of the form $[\pi,\infty)$. By the $2\pi$-periodicity of $\exp(\theta\xi_3)$, every element of
$\SO$ is a conjugate point. Conversely, the preceding characterization
of the conjugate times shows that every conjugate point belongs to $\SO$.
Therefore,  $ \mathcal{Con}(\Id) = \SO$. 
\end{proof}

\section{Global Study}\label{sec:glo}

In this section, we study Maxwell points arising from intersections of geodesics with initial velocities in $\Gamma_\omega$ and $\Gamma_{-\omega}$. In Subsection \ref{subsec:max-time}, we reduce the Maxwell condition to a scalar equation in each of the three regimes determined by $\omega$ and use these equations to determine the first Maxwell time, proving Theorem \ref{thm:Max-time}. We finish this subsection by proving Corollary \ref{cor:t-max-t-con} and Lemma \ref{lem:g-ome-sym}. In Subsection \ref{subsec:reg-fir-max}, we study how $\Tmax(\omega)$ depends on $\omega$, including its regularity and global minimum, and prove Proposition 1.4.

\subsection{Maxwell Times}\label{subsec:max-time}

We begin by determining the Maxwell times corresponding to the $\Gamma_\omega$-symmetry. We then compare the first Maxwell time with the first conjugate time.

Our strategy is as follows. Consider the geodesic $\gamma_\omega$. We seek a time $t_\omega >0$ and a geodesic $\gamma_1\vcentcolon\R \to \SLg$ satisfying $\gamma_1(0) = \Id$, $\dot{\gamma}_1(0)\in \Gamma_{-\omega}$, $\dot{\gamma}_1(0) \neq \xi_\omega$,  such that $\gamma_\omega(t_\omega) = \gamma_1(t_\omega)$. To this end, we take the ansatz
\begin{equation}\label{eq:ansatz}
    \gamma_1(0) = \Id, \quad\text{and}\quad \dot{\gamma}_1(0) = \sqrt{1-\omega^2}\big(\cos\theta \xi_1 - \sin\theta \xi_2\big) - \omega \xi_3.
\end{equation}
     The parameter $\theta \in [0,2\pi)$ represents the angular coordinate of the initial velocity in $\Gamma_{-\omega}$. We determine $\theta$ by imposing the intersection condition $\gamma_\omega(t_\omega) = \gamma_1(t_\omega)$.  For $\omega\neq 0$, we divide the analysis into three cases: \textbf{(1)} $\omega \in [-1,-\frac{1}{\sqrt{2}}) \cup(\frac{1}{\sqrt{2}},1]$; \textbf{(2)} $\omega = \pm \frac{1}{\sqrt{2}}$; and \textbf{(3)} $\omega \in (-\frac{1}{\sqrt{2}},0) \cup(0,\frac{1}{\sqrt{2}})$. Since the methodology is identical for all three cases, we present only the first one in detail.

\subsubsection{Case $\omega \in [-1,-\frac{1}{\sqrt{2}}) \cup(\frac{1}{\sqrt{2}},1] $}

\begin{lemma}\label{lem:up-bound-1}
    Let $\omega \in [-1,-\frac{1}{\sqrt{2}}) \cup(\frac{1}{\sqrt{2}},1]$, and let $\gamma_\omega$ be the geodesic defined in Eq. \eqref{eq:def-gam-ome}. The Maxwell times corresponding to the $\Gamma_\omega$-symmetry are precisely the positive solutions of
        \begin{equation*}
            \cos(t\sqrt{2\omega^2-1}) \sin(2\omega t)  -\frac{\omega}{\sqrt{2\omega^2-1}} \sin(t\sqrt{2\omega^2-1})\cos(2\omega t) = 0.
        \end{equation*}
\end{lemma}

\begin{proof}
     Consider the ansatz in Eq. \eqref{eq:ansatz}. To determine $\theta$, we first use Lemma \ref{lem:expl-form} and the identities
$$ \exp(t(\xi_\omega-\xi_\omega^\T)) = \exp(2\omega\xi_3 t) = \cos(2\omega t) \Id +\sin(2\omega t) \xi_3,\qquad \xi_1\xi_3 = -\xi_2,$$
to express the difference $\gamma_\omega(t)-\gamma_1(t)$ as a linear combination of $\xi_1$, $\xi_2$, and $\xi_3$. Since the initial velocities of both geodesics have the same determinant, the coefficients of $\Id$ in their expressions given by Lemma \ref{lem:expl-form} coincide. Consequently, a direct computation gives
     \begin{equation}\label{eq:cut-point-proof}
         \gamma_\omega(t) -\gamma_1(t) = \psi_\theta(t)\xi_1+\phi_\theta(t)\xi_2 +2F^1(t,\omega)\xi_3.
     \end{equation}
     Here $\psi_\theta,\phi_\theta\vcentcolon\R \to \R$ are auxiliary functions defined by
     \begin{equation*}
        \begin{split}
            \psi_\theta(t) &\coloneqq -\frac{2 \sqrt{1-\omega ^2} \sin\!\big(\frac{\theta}{2}\big)\sin\!\big(\frac{1}{2} (4\omega t-\theta )\big) \sin\!\big(t \sqrt{2\omega^2-1}\big)}{\sqrt{2\omega^2-1}}, \\
            \phi_\theta(t) &\coloneqq -\frac{2 \sqrt{1-\omega ^2} \cos\!\big(\frac{\theta}{2}\big)\sin\!\big(\frac{1}{2} (4\omega t-\theta )\big) \sin\!\big(t \sqrt{2\omega^2-1}\big)}{\sqrt{2\omega^2-1}}.  
         \end{split}
     \end{equation*}
The auxiliary function $F^1\vcentcolon\R\times ([-1,-\frac{1}{\sqrt{2}}) \cup(\frac{1}{\sqrt{2}},1])\to \R$ is defined by
     \begin{equation}\label{eq:def-f-function}
        \begin{split}
            F^1(t,\omega) \coloneqq \cos(t\sqrt{2\omega^2-1}) \sin(2\omega t)  -\frac{\omega}{\sqrt{2\omega^2-1}} \sin(t\sqrt{2\omega^2-1})\cos(2\omega t). 
         \end{split}
     \end{equation}
     For the moment, assume there exists $(t_\omega,\omega) \in (0,\infty)\times ([-1,-\frac{1}{\sqrt{2}}) \cup(\frac{1}{\sqrt{2}},1]) $ such that $F^1(t_\omega,\omega) = 0$. We then choose $\theta \equiv 4\omega t_\omega \pmod{2\pi}$. By the expressions for $\psi_\theta$ and $\phi_\theta$, this choice ensures that $\psi_\theta(t_\omega) = \phi_\theta(t_\omega) = 0$. Consequently, Eq \eqref{eq:cut-point-proof} reduces to 
     \begin{equation*}
    \begin{split}
      \gamma_\omega(t_\omega) -\gamma_1(t_\omega) = 2 F^1(t_\omega,\omega) \xi_3 = 0. 
    \end{split}
    \end{equation*}
    Since $\omega\neq 0$, the initial velocities of $\gamma_\omega$ and $\gamma_1$ are distinct. Consequently, their intersection at $t_\omega$ is a Maxwell point.
     
    Conversely, suppose that $t>0$ is a Maxwell time corresponding to the $\Gamma_\omega$-symmetry. Then there exists a geodesic $\gamma_1$ with
initial velocity in $\Gamma_{-\omega}$ such that $\gamma_\omega(t)=\gamma_1(t).$
By Eq. \eqref{eq:cut-point-proof} and the linear independence of $\xi_1,\xi_2,\xi_3$, we obtain $F^1(t,\omega)=0.$ Therefore, the Maxwell times corresponding to the $\Gamma_\omega$-symmetry are precisely the positive zeros of $F^1(\cdot,\omega)$.
    
    Finally, we show that the equation $F^1(t_\omega,\omega) = 0$ admits a positive solution. By the symmetry $F^1(t,-\omega)= - F^1(t,\omega)$, it suffices to consider $\omega>0$. Define $f_\omega^1(t) \coloneqq F^1(t,\omega)$. Since $\omega \in(\frac{1}{\sqrt{2}},1]$, we have $0<\frac{\pi\sqrt{2\omega^2-1}}{2\omega} \leq \frac{\pi}{2}$. Consequently,
    $$f_\omega^1\Big(\frac{\pi}{2\omega}\Big) = \frac{\omega}{\sqrt{2\omega^2-1}}\sin\!\Big(\frac{\pi \sqrt{2\omega^2-1}}{2\omega}\Big) > 0,$$
    and
    $$f_\omega^1\Big(\frac{\pi}{\omega}\Big)= -\frac{\omega}{\sqrt{2\omega^2-1}}\sin\!\Big(\frac{\pi \sqrt{2\omega^2-1}}{\omega}\Big) \leq 0, $$
    where equality holds if and only if $\omega=1$. By the intermediate value theorem, there exists $t_\omega\in (\frac{\pi}{2\omega},\frac{\pi}{\omega}]$ such that $f^1_\omega(t_\omega)= 0$.  
\end{proof}

\begin{lemma}\label{lem:fir-max-tim-1}
    For each $\omega\in [-1,-\frac{1}{\sqrt{2}}) \cup(\frac{1}{\sqrt{2}},1]$, the first positive zero $t_\omega^*$ of $f^1_\omega$ satisfies   $t_\omega^* \in (\frac{\pi}{2|\omega|},\frac{\pi}{|\omega|}]$. Furthermore, $t_\omega^* = \frac{\pi}{|\omega|}$ if and only if $\omega = \pm 1$.   In particular, $t^*_\omega = \Tmax(\omega)$.
 
\end{lemma}


\begin{proof}
    By the symmetry $F^1(t,-\omega)= - F^1(t,\omega)$, it suffices to consider $\omega\in(\frac{1}{\sqrt{2}},1]$. For $t \in (0,\frac{\pi}{2\omega})$, we have 
    $$ 0 <t\sqrt{2\omega^2-1} < t\omega < \frac{\pi}{2},$$ 
    where we used the inequality $2\omega^2-1 \leq \omega^2$. Consequently, $\sin(t\sqrt{2\omega^2-1})>0$ and $\cos(t\sqrt{2\omega^2-1})>0$. We now split the proof into two cases. 

    Case $t\in (0,\frac{\pi}{4|\omega|})$. Since $\sqrt{2\omega^2-1}\leq \omega$, we have $0<t\sqrt{2\omega^2-1}\leq 2t\omega<\frac{\pi}{2}$. The function $x \to \frac{\tan x}{x}$ is strictly increasing on $(0,\frac{\pi}{2})$. Hence, 
    $$ \frac{\tan(2\omega t)}{2\omega t} > \frac{\tan(t\sqrt{2\omega^2-1})}{t\sqrt{2\omega^2-1}}\quad \Rightarrow \quad \tan(2\omega t)> \frac{\omega}{\sqrt{2\omega^2-1}} \tan(t\sqrt{2\omega^2-1}). $$
    Therefore, $f_\omega^1(t)>0$.


    Case $t\in [\frac{\pi}{4|\omega|},\frac{\pi}{2|\omega|}]$. Here $\sin(2\omega t) >0$ and  $\cos(2\omega t) \leq 0$. Therefore, both terms in $f^1_\omega(t)$ are nonnegative, and at least one is strictly positive. Hence $f^1_\omega(t)>0$.  

    If $\omega = 1$, then $f^1_\omega(t) = \sin t$. Hence the first positive zero is $t_\omega^* = \pi$. 
\end{proof}

\subsubsection{Case $\omega = \pm \frac{1}{\sqrt{2}} $}

\begin{lemma}\label{lem:up-bound-2}
    Let $\omega \in  \{\pm \frac{1}{\sqrt{2}}\}$, and let $\gamma_\omega$ be the geodesic defined in Eq. \eqref{eq:def-gam-ome}. 
    The Maxwell times corresponding to the $\Gamma_\omega$-symmetry are precisely the positive solutions of
        \begin{equation*}
            \sin(2 \omega t) -  \omega t \cos( 2 \omega t) = 0.
        \end{equation*}
\end{lemma}

\begin{proof}
     Using the same strategy as in the proof of Lemma \ref{lem:up-bound-1}, we consider the following: 
    \begin{equation}\label{eq:def-g-function}
     \begin{split}
           F^2(t,\omega) \coloneqq  \sin(2 \omega  t) - \omega t \cos(2 \omega  t) .
     \end{split}
    \end{equation}
     Let $\gamma_\omega(t)$ be the geodesic defined in Eq. \eqref{eq:def-gam-ome}, and let $\gamma_2(t)$ be the geodesic  with initial conditions
    $$\gamma_2(0) = \Id, \quad\text{and}\quad \dot{\gamma}_2(0) = |\omega|\big(\cos(4 \omega t_\omega) \xi_1 - \sin(4 \omega t_\omega) \xi_2\big) - \omega \xi_3 . $$
    Since $\omega\neq 0$, these geodesics have distinct initial velocities. We show that the condition  $F^2(t_\omega,\omega) = 0$ implies $\gamma_\omega(t_\omega) = \gamma_2(t_\omega)$. Indeed,  using the explicit formulas provided by Lemma \ref{lem:expl-form} to compute  the matrices $\gamma_\omega(t_\omega)$ and $\gamma_2(t_\omega)$, we find that
    \begin{equation*}
    \begin{split}
      \gamma_\omega(t_\omega) -\gamma_2(t_\omega) = 2 F^2(t_\omega,\omega) \xi_3.
    \end{split}
    \end{equation*}
    Conversely, suppose that $t>0$ is a Maxwell time corresponding to the $\Gamma_\omega$-symmetry. By Lemma \ref{lem:expl-form}, the coefficient of $\xi_3$ in the difference between $\gamma_\omega$ and any geodesic with initial velocity in $\Gamma_{-\omega}$ is $2F^2(t,\omega)$. Since the geodesics intersect at time $t$, this coefficient must vanish. Therefore, $F^2(t,\omega)=0$.

    Finally, we show that the equation $F^2(t_\omega,\omega) = 0$ admits a positive solution. By the symmetry $F^2(t,-\omega)= - F^2(t,\omega)$, it suffices to consider $\omega = \frac{1}{\sqrt{2}}$. Define $f^2_\omega(t) \coloneqq F^2(t,\omega)$. In this case,
    $$f^2_\omega\bigg(\frac{\pi}{\sqrt{2}}\bigg) = \frac{\pi}{2}\;\; \text{and} \;\;f^2_\omega(\sqrt{2} \pi) = -\pi.$$
    Hence, the intermediate value theorem guarantees a solution $t_\omega\in\left(\frac{\pi}{2\omega},\frac{\pi}{\omega}\right)$.
\end{proof}    

\begin{lemma}\label{lem:fir-max-tim-2}
    For each $\omega \in  \{\pm \frac{1}{\sqrt{2}}\}$, the first positive zero $t_\omega^*$ of   $f^2_\omega$ satisfies $t_\omega^* \in (\frac{\pi}{2|\omega|},\frac{\pi}{|\omega|})$. In particular, $t^*_\omega = \Tmax(\omega)$.
\end{lemma}

\begin{proof}
    By the symmetry $F^2(t,-\omega)= - F^2(t,\omega)$, it suffices to consider $\omega= \frac{1}{\sqrt{2}}$. We split the interval into two cases.
    
    Case $t\in (0,\frac{\pi}{4\omega}]$. Since $\tan x>x$ for $x\in (0,\frac{\pi}{2})$, it follows that $\tan(2\omega t) > \omega t$. Multiplying by $\cos(2\omega t)>0$, we obtain $\sin(2\omega t) > \omega t \cos(2\omega t)$. Therefore, $f_\omega^2(t)>0$. 

    Case $t\in (\frac{\pi}{4\omega},\frac{\pi}{2\omega})$. Here $\sin(2\omega t) >0$ and  $\cos(2\omega t) \leq 0$. Therefore, both terms in $f^2_\omega(t)$ are nonnegative, and at least one is strictly positive. Hence $f^2_\omega(t)>0$.
\end{proof}

\subsubsection{Case $\omega \in (-\frac{1}{\sqrt{2}},\frac{1}{\sqrt{2}}) $}

\begin{lemma}\label{lem:up-bound-3}
     Let $\omega \in (-\frac{1}{\sqrt{2}},0) \cup (0,\frac{1}{\sqrt{2}})$, and let $\gamma_\omega$ be the geodesic defined in Eq. \eqref{eq:def-gam-ome}. The Maxwell times corresponding to the $\Gamma_\omega$-symmetry are precisely the positive solutions of
        \begin{equation*}
            \cosh(t\sqrt{1-2\omega^2}) \sin(2\omega t)  -\frac{\omega}{\sqrt{1-2\omega^2}} \sinh(t\sqrt{1-2\omega^2})\cos(2\omega t) = 0.
        \end{equation*}
\end{lemma}

\begin{proof}
     We again follow the same strategy as in the proof of Lemma \ref{lem:up-bound-1}. Consider the following function
    \begin{equation}\label{eq:def-h-function}
     \begin{split}
           F^3(t,\omega) \coloneqq \cosh(t\sqrt{1-2\omega^2}) \sin(2\omega t)  -\frac{\omega}{\sqrt{1-2\omega^2}} \sinh(t\sqrt{1-2\omega^2})\cos(2\omega t).
     \end{split}
    \end{equation}
    Let $\gamma_\omega(t)$ be the geodesic defined in Eq. \eqref{eq:def-gam-ome}, and let $\gamma_3(t)$ be the geodesic with initial condition 
    $$\gamma_3(0) = \Id, \quad\text{and}\quad \dot{\gamma}_3(0) = \sqrt{1-\omega^2}\big(\cos(4\omega t_\omega) \xi_1 - \sin(4\omega t_\omega) \xi_2\big) - \omega\xi_3. $$
   We must show that the condition $F^3(t_\omega,\omega) = 0$ implies $\gamma_\omega(t_\omega) = \gamma_3(t_\omega)$. Indeed, using the explicit formulas provided by Lemma  \ref{lem:expl-form} to compute  the matrices $\gamma_\omega(t_\omega)$ and $\gamma_3(t_\omega)$, we find that
    \begin{equation*}
    \begin{split}
      \gamma_\omega(t_\omega) -\gamma_3(t_\omega) = 2F^3(t_\omega,\omega) \xi_3
    \end{split}
    \end{equation*}
    Conversely, suppose that $t>0$ is a Maxwell time corresponding to the $\Gamma_\omega$-symmetry. By Lemma \ref{lem:expl-form}, the coefficient of $\xi_3$ in the difference between $\gamma_\omega$ and any geodesic with initial velocity in $\Gamma_{-\omega}$ is $2F^3(t,\omega)$. Since the geodesics intersect at time $t$, this coefficient must vanish. Therefore, $F^3(t,\omega)=0$.

    Finally, we show that the equation $F^3(t_\omega,\omega) = 0$ admits a positive solution. By the symmetry $F^3(t,-\omega)= - F^3(t,\omega)$, it suffices to consider $\omega \in (0, \frac{1}{\sqrt{2}})$. Define $f^3_\omega(t)\coloneqq F^3(t,\omega)$. In this case,
    $$ F^3_\omega(\frac{\pi}{2\omega})>0\;\; \text{and}\;\;F^3_\omega(\frac{\pi}{\omega})<0 .$$
    Hence, the Intermediate Value Theorem guarantees a solution $t_\omega\in(\frac{\pi}{2\omega},\frac{\pi}{\omega})$.
\end{proof}

\begin{lemma}\label{lem:fir-max-tim-3}
    For each $\omega\in (-\frac{1}{\sqrt{2}},0) \cup(0,\frac{1}{\sqrt{2}})$, the first positive zero $t_\omega^*$ of  $f^3_\omega$ satisfies $t_\omega^* \in (\frac{\pi}{2|\omega|},\frac{\pi}{|\omega|})$.    In particular, $t^*_\omega = \Tmax(\omega)$.
\end{lemma}

\begin{proof}
    First observe that $F^3(t,-\omega)= - F^3(t,\omega)$. Thus, it suffices to focus on the case  $\omega\in(0,\frac{1}{\sqrt{2}})$. We now split the proof into two cases.

    Case $t\in (0,\frac{\pi}{4\omega})$. In this interval $0<2\omega t< \frac{\pi}{2}$. Therefore, $\sin(2\omega t) > 0$ and $\cos(2\omega t) > 0$. Also, $\cosh(t\sqrt{1-2\omega^2})>0$ and $\sinh(t\sqrt{1-2\omega^2})>0$. Then $F^3_\omega(t) > 0$ is equivalent to 
    $$ \tan(2\omega t) > \frac{\omega}{\sqrt{1-2\omega^2}}\tanh(t\sqrt{1-2\omega^2}).$$
    The inequality $0<2\omega t < \frac{\pi}{2}$ implies $\tan(2\omega t) > 2\omega t$, as well as $\tanh(t\sqrt{1-2\omega^2}) < t\sqrt{1-2\omega^2}$. Therefore,
    $$ \tan(2\omega t) > \omega t > \frac{\omega}{\sqrt{1-2\omega^2}}\tanh(t\sqrt{1-2\omega^2}).   $$
    Hence $f^3_\omega(t) > 0$ for $t \in (0,\frac{\pi}{4\omega})$.

    Case $t \in [\frac{\pi}{4\omega},\frac{\pi}{2\omega}]$. Here $\sin(2\omega t) >0$ and  $\cos(2\omega t) \leq 0$. Therefore, both terms in $f^3_\omega(t)$ are nonnegative, and at least one is strictly positive. Hence $f^3_\omega(t)>0$.
\end{proof}


\begin{proof}[Proof of Theorem \ref{thm:Max-time}] We consider separately the cases  $\omega \in [-1,-\frac{1}{\sqrt{2}})\cup(\frac{1}{\sqrt{2}},1]$, $\omega = \pm \frac{1}{\sqrt{2}}$ and $\omega \in (-\frac{1}{\sqrt{2}},0)\cup (0,\frac{1}{\sqrt{2}})$.  

\textbf{Case 1}: $\omega \in [-1,-\frac{1}{\sqrt{2}})\cup(\frac{1}{\sqrt{2}},1]$). Consider the geodesic $\gamma_\omega(t)$ defined in Eq. \eqref{eq:def-gam-ome}. Lemma \ref{lem:up-bound-1} implies that the Maxwell times corresponding to $\Gamma_\omega$-symmetry are precisely the positive zeros of $ f_\omega^1$. By Lemma \ref{lem:fir-max-tim-1}, its first positive zero $t_\omega^*$ satisfies $t_\omega^* \in (\frac{\pi}{2|\omega|},\frac{\pi}{|\omega|}]$, with equality at the upper endpoint if and only if \(\omega=\pm1\). By the $\SO$-symmetry and Proposition \ref{prp:iso-action}, all geodesics with initial velocities in $\Gamma_\omega$ have the same first Maxwell time. Consequently, $\Tmax(\omega) = t_\omega^*$. 

The remaining two cases follow by the same argument, using Lemmas \ref{lem:up-bound-2}-\ref{lem:fir-max-tim-3}. 
\end{proof}

\begin{proof}[Proof of Corollary \ref{cor:t-max-t-con}]
    The result is immediate for $\omega \in [-\frac{1}{\sqrt{2}},0)\cup (0,\frac{1}{\sqrt{2}}]$, since $\Tcon (\omega) = \infty$. Thus, it suffices to consider the case $\omega \in [-1,-\frac{1}{\sqrt{2}}) \cup (\frac{1}{\sqrt{2}},1]$. If $\omega \in (-1,-\frac{1}{\sqrt{2}}) \cup (\frac{1}{\sqrt{2}},1)$, then
    $$\Tmax(\omega) <  \frac{\pi}{|\omega|} < \frac{\pi}{\sqrt{2\omega^2-1}} = \Tcon (\omega) .$$
    If $\omega = \pm 1$, then $\Tmax(\omega) = \Tcon (\omega)$.
\end{proof}

\subsection{Regularity of the First Maxwell Time}\label{subsec:reg-fir-max}

\begin{lemma}\label{lem:dt-max-dome-1}
For $\omega\in (\frac{1}{\sqrt{2}},1)$, let $\Tmax(\omega)$ denote the first positive solution of $F^1(t,\omega) = 0$. Then the following statements hold.
    
    If $\Tmax(\omega)\sqrt{2\omega^2-1}\neq \frac{\pi}{2}$, then
\begin{equation}\label{eq:dt-mx-domega-1}
    \frac{d}{d\omega}\Tmax
=
\frac{
2\sqrt{2\omega^2-1}(1-\omega^2)\Tmax(\omega)
-
\tan\!\left(\Tmax(\omega)\sqrt{2\omega^2-1}\right)
}{
\omega\sqrt{2\omega^2-1}
\left[
2\omega^2-1+
\tan^2\!\left(\Tmax(\omega)\sqrt{2\omega^2-1}\right)
\right]
}.
\end{equation}

If $\Tmax(\omega)\sqrt{2\omega^2-1} = \frac{\pi}{2}$, then
\begin{equation}\label{eq:dt-mx-domega-1-1}
    \frac{d}{d\omega}\Tmax
= 0. 
\end{equation}
\end{lemma}

\begin{proof}
    We begin by considering $\Tmax(\omega)\sqrt{2\omega^2-1}\neq \frac{\pi}{2}$. If $\cos(2\omega t) = 0$, then we have $\cos(\Tmax(\omega)\sqrt{2\omega^2-1}) = 0$. So under the stated assumption, both cosine factors are nonzero, and the equation $F^1(\Tmax(\omega),\omega) = 0$ is equivalent to $G^1(\Tmax(\omega),\omega) = 0$ where
    $$ G^1 (t,\omega) \coloneqq  \tan(2\omega t)-\frac{\omega}{\sqrt{2\omega^2-1}}\tan(t\sqrt{2\omega^2-1}). $$
    Computing the partial derivatives, we obtain
    \begin{equation*}
        \begin{split}
      G^1_t & = \omega\big(2\sec^2(2\omega t) - \sec^2(t\sqrt{2\omega^2-1})\big),\\
      G^1_\omega  & = 2t\sec^2(2\omega t) + \frac{1}{(2\omega^2-1)^{\frac{3}{2}}}\tan(t\sqrt{2\omega^2-1}) - \frac{2\omega^2t}{2\omega^2-1}  \sec^2(t\sqrt{2\omega^2-1}).
        \end{split}
    \end{equation*}
    Using the condition $G^1(\Tmax(\omega),\omega) = 0$ and the identity \(\sec^2(x)=1+\tan^2(x)\), then
    \begin{equation*}
        \begin{split}
        \sec^2(\Tmax(\omega)\sqrt{2\omega^2-1}) &= 1+\tan^2(\Tmax(\omega)\sqrt{2\omega^2-1}), \\
        \sec^2(2\omega \Tmax(\omega)) & = 1 + \frac{\omega^2}{2\omega^2-1}\tan^2(\Tmax(\omega)\sqrt{2\omega^2-1}).
        \end{split}
    \end{equation*}
    Thus
    \begin{equation*}
        \begin{split}
            G^1_t & = \frac{\omega}{2\omega^2-1} \big( 2\omega^2-1+\tan^2(\Tmax(\omega)\sqrt{2\omega^2-1})\big)>0,\\
            G^1_\omega & = -\frac{2(1-\omega^2)\Tmax(\omega)}{2\omega^2-1} + \frac{1}{(2\omega^2-1)^{\frac{3}{2}}}\tan(\Tmax(\omega)\sqrt{2\omega^2-1}).
        \end{split}
    \end{equation*}
    Therefore, implicit differentiation yields $\frac{d}{d\omega}\Tmax = - \frac{G^1_\omega}{G^1_t}$ which gives Eq. \eqref{eq:dt-mx-domega-1}. 

     We now prove Eq. \eqref{eq:dt-mx-domega-1-1}. If $\Tmax(\omega)\sqrt{2\omega^2-1}= \frac{\pi}{2}$, then $F^1(\Tmax(\omega),\omega) = 0$ is equivalent to 
     $$ -\frac{\omega}{\sqrt{2\omega^2-1}}  \cos(2\omega \Tmax(\omega)) = 0,$$
     and therefore $\cos(2\omega \Tmax(\omega)) = 0$. Hence $\sin(2\omega \Tmax(\omega)) = \pm 1$. We now evaluate the partial derivatives of $F^1$ at this point. 
     \begin{equation*}
        \begin{split}
      F^1_t  & = \sin(2\omega \Tmax(\omega)) \big(-\sqrt{2\omega^2-1} + \frac{2\omega^2}{\sqrt{2\omega^2-1}}\big) = \pm\frac{1}{\sqrt{2\omega^2-1}} \neq 0,\\
      F^1_\omega & = \sin(2\omega \Tmax(\omega))\big(-\Tmax(\omega) \frac{\partial}{\partial \omega }\sqrt{2\omega^2-1} + \frac{2\omega \Tmax(\omega)}{\sqrt{2\omega^2-1}}\big) =0 .
        \end{split}
    \end{equation*}
    Therefore, implicit differentiation yields $\frac{d}{d\omega}\Tmax = - \frac{F^1_\omega}{F^1_t}$ which gives Eq. \eqref{eq:dt-mx-domega-1-1}.
\end{proof}

\begin{Prop}\label{prp:dt-dome-neg-1}
    The function $\Tmax$ has a unique minimum on $(\frac{1}{\sqrt{2}},1)$ attained at $\omega = \frac{3}{\sqrt{14}}$ with $\Tmax(\frac{3}{\sqrt{14}}) = \frac{\pi\sqrt{14}}{4}$.
\end{Prop}

\begin{proof}
    Lemma \ref{lem:fir-max-tim-1} implies that $\Tmax(\omega) \in (\frac{\pi}{2\omega},\frac{\pi}{\omega})$ for $\omega\in(\frac{1}{\sqrt{2}},1)$. As we see in the proof of Lemma \ref{lem:dt-max-dome-1}, if $\Tmax(\omega)\sqrt{2\omega^2-1}= \frac{\pi}{2}$, then the condition $F^1(\Tmax(\omega),\omega) = 0$ implies that $\cos(2\omega \Tmax(\omega)) = 0$. Hence $\pi<2\omega \Tmax(\omega)<2\pi$. The only zero of cosine in this interval is $2\omega \Tmax(\omega) = \frac{3\pi}{2}$. Combining this identity with $\Tmax(\omega)\sqrt{2\omega^2-1}= \frac{\pi}{2}$, we obtain
    $$\frac{\omega}{\sqrt{2\omega^2-1}} = \frac{3}{2},$$
    which leads to $\omega = \frac{3}{\sqrt{14}}$. 

    If $\omega\in (\frac{1}{\sqrt{2}},1)$ and $0<\Tmax(\omega)\sqrt{2\omega^2-1} < \frac{\pi}{2}$, then since $2(1-\omega^2)<1$ and
$\tan x>x$ for $x\in(0,\pi/2)$, we obtain
    $$ 2(1-\omega^2)\Tmax(\omega)\sqrt{2\omega^2-1}<\Tmax(\omega)\sqrt{2\omega^2-1}<\tan \big(\Tmax(\omega)\sqrt{2\omega^2-1}\big). $$
    Thus, Eq. \eqref{eq:dt-mx-domega-1}  implies $\frac{d}{d\omega} \Tmax <0$.

    If $\omega\in (\frac{1}{\sqrt{2}},1)$ and $\frac{\pi}{2}<\Tmax(\omega)\sqrt{2\omega^2-1} < \pi$, then $\tan \big(\Tmax(\omega)\sqrt{2\omega^2-1}\big)< 0$. It follows that 
    $$  2\sqrt{2\omega^2-1}(1-\omega^2)\Tmax(\omega)-\tan \big(\Tmax(\omega)\sqrt{2\omega^2-1}\big)>0. $$
    Thus, Eq. \eqref{eq:dt-mx-domega-1}  implies $\frac{d}{d\omega} \Tmax >0$.
\end{proof}

\begin{lemma}\label{lem:dt-max-dome-2}
    For $\omega\in (0,\frac{1}{\sqrt{2}})$,  let $\Tmax(\omega)$ denote the first positive solution of $F^3(t,\omega) = 0$. Then the following statement holds.
\begin{equation}\label{eq:dt-mx-domega-2}
    \frac{d}{d\omega}\Tmax
=
\frac{\tanh\!\left(\Tmax(\omega)\sqrt{1-2\omega^2}\right)- 2\sqrt{1-2\omega^2}(1-\omega^2)\Tmax(\omega)}{
\omega\sqrt{1-2\omega^2}
\left[
1-2\omega^2+
\tanh^2\!\left(\Tmax(\omega)\sqrt{1-2\omega^2}\right)
\right]
}.
\end{equation}
\end{lemma}

\begin{proof}
    We first claim that \(\cos(2\omega \Tmax(\omega))\neq0\) for $\omega \in (0,\frac{1}{\sqrt{2}})$. Indeed, if  \(\cos(2\omega \Tmax(\omega)) = 0\), then $F^3(\Tmax(\omega),\omega) = 0$ implies 
    $$\cosh\big(\Tmax(\omega)\sqrt{1-2\omega^2}\big) = 0,$$
    which is a contradiction. 

    Therefore, we can consider the function 
    $$ G^3 (t,\omega) \coloneqq  \tan(2\omega t)-\frac{\omega}{\sqrt{1-2\omega^2}}\tanh(t\sqrt{1-2\omega^2}). $$
    Along the Maxwell branch, $G^3_t>0$. Hence, the implicit function theorem applies, and the same calculation as in Lemma \ref{lem:dt-max-dome-1} yields \eqref{eq:dt-mx-domega-2}.
\end{proof}

\begin{Prop}\label{prp:dt-dome-neg-2}
    The function $\Tmax$ is strictly decreasing on $(0,\frac{1}{\sqrt{2}})$.
\end{Prop}
\begin{proof}
    For $\omega\in (0,\frac{1}{\sqrt{2}})$, we have $2(1-\omega^2) > 1$. This inequality, together with
    $$ \tanh\big( \Tmax(\omega)\sqrt{1-2\omega^2}\big) < \Tmax(\omega)\sqrt{1-2\omega^2}$$
    implies that the numerator in Eq. \eqref{eq:dt-mx-domega-2} is negative. Therefore, $\frac{d}{d\omega} \Tmax < 0$. 
\end{proof}

\begin{proof}[Proposition \ref{prp:f-max-t-reg}]
    Propositions \ref{prp:dt-dome-neg-1} and \ref{prp:dt-dome-neg-2} imply that $\Tmax$ \(t_{\mathrm{Max}}\) is continuously differentiable on $(0,\frac{1}{\sqrt{2}})\cup(\frac{1}{\sqrt{2}},1)$. We now establish its $C^1$-regularity at $\omega = \frac{1}{\sqrt{2}}$. By Lemmas \ref{lem:fir-max-tim-1}, \ref{lem:fir-max-tim-2}, and \ref{lem:fir-max-tim-3}, the first Maxwell time satisfies $\Tmax(\omega) \in (\frac{\pi}{2\omega},\frac{\pi}{\omega}]$ for $\omega>0$. Moreover,
    $$ \lim_{\omega \to {\frac{1}{\sqrt{2}}^+}}F^1(t,\omega) = F^2(t,\omega) = \lim_{\omega \to {\frac{1}{\sqrt{2}}^-}}F^3(t,\omega).$$
    Moreover, $F^2(\cdot,\frac{1}{\sqrt{2}})$ has a unique zero in $[\frac{\pi}{\sqrt{2}},\sqrt{2}\pi]$, as follows from the strict monotonicity of $\tan x -\frac{x}{2}$ on $(\pi,\frac{3}{2}\pi)$. Then, every accumulation point of $\Tmax(\omega)$, as $\omega \to \frac{1}{\sqrt{2}}$, must equal $\Tmax(\frac{1}{\sqrt{2}})$. Therefore, $\Tmax$ is continuous at $\omega = \frac{1}{\sqrt{2}}$.

    Since $\Tmax$ is continuous at $\omega = \frac{1}{\sqrt{2}}$ and continuously differentiable on either side, it suffices to show that its one-sided derivatives converge to the same finite limit.
    $$ \lim_{\omega \to {\frac{1}{\sqrt{2}}^+}} \frac{d}{d\omega}\Tmax(\omega) = \lim_{\omega \to {\frac{1}{\sqrt{2}}^-}} \frac{d}{d\omega}\Tmax(\omega).$$

    Consider $\omega\in (\frac{1}{\sqrt{2}},1)$ and set $\varepsilon = \sqrt{2\omega^2-1}$. Using the identity $2(1-\omega^2) = 1 -\varepsilon^2$ and the Taylor expansion $\tan x = x + \frac{x^3}{3}+O(x^5)$, we obtain the following expansion for the numerator in Eq.\eqref{eq:dt-mx-domega-1}:
    $$ \varepsilon(1-\varepsilon^2) \Tmax -\varepsilon \Tmax - \frac{\varepsilon^2\Tmax^3}{3}+O(\varepsilon^5) = -\varepsilon^3 \Tmax(1+\frac{\Tmax^2}{3}) +O(\varepsilon^5).$$
    Similarly, the denominator satisfies
    $$ \varepsilon^3\omega[1+\Tmax^2+O(\varepsilon^2)].$$
    Therefore, 
    $$\lim_{\omega \to {\frac{1}{\sqrt{2}}^+}} \frac{d}{d\omega}\Tmax = -\sqrt{2}\Tmax\bigg(\frac{1}{\sqrt{2}}\bigg)\frac{1+\frac{\Tmax^2(\frac{1}{\sqrt{2}})}{3}}{1+\Tmax^2(\frac{1}{\sqrt{2}})}.$$

    Now consider $\omega\in (0,\frac{1}{\sqrt{2}})$, and set $\varepsilon = \sqrt{1-2\omega^2}$. The numerator in Eq. \eqref{eq:dt-mx-domega-2} satisfies
    $$ - \varepsilon^3\Tmax (1+\frac{\Tmax^2}{3}) + O(\varepsilon^5), $$
    while the denominator satisfies
    $$ \varepsilon^3\omega[1+\Tmax^2+O(\varepsilon^2)]. $$
    Therefore, 
    $$\lim_{\omega \to {\frac{1}{\sqrt{2}}^-}} \frac{d}{d\omega}\Tmax = -\sqrt{2}\Tmax\bigg(\frac{1}{\sqrt{2}}\bigg)\frac{1+\frac{\Tmax^2(\frac{1}{\sqrt{2}})}{3}}{1+\Tmax^2(\frac{1}{\sqrt{2}})}.$$
Thus, the one-sided derivatives agree, and $\Tmax$ is $C^1$ at $\omega = \frac{1}{\sqrt{2}}$. 

Finally, Propositions \ref{prp:dt-dome-neg-1} and \ref{prp:dt-dome-neg-2} imply that $\Tmax$ attains its unique global minimum at
    $\omega = \frac{3}{\sqrt{14}}$ and $\Tmax(\frac{3}{\sqrt{14}}) = \frac{\pi\sqrt{14}}{4}. $ 
\end{proof}

Having characterized the first Maxwell time in each of the three regimes and established the continuity of $\Tmax(\omega)$, we now prove Lemma \ref{lem:g-ome-sym}, which describes the geometry of the Maxwell stratum.

\begin{proof}[Lemma \ref{lem:g-ome-sym}] Consider first $\omega\in (\frac{1}{\sqrt{2}},1)$. Using Lemma \ref{lem:expl-form} and the same computation as in Eq. \eqref{eq:cut-point-proof}, we obtain
$$ \gamma_\omega(t)-\gamma_\omega(t)^\T = 2F^1(t,\omega)\xi_3.$$
Since $F^1(\Tmax(\omega),\omega)=0$, the preceding identity implies that $g_\omega$ is symmetric. The same argument applies in the remaining two regimes. Thus, $g_\omega$ is symmetric for all $\omega\in (0,1]$. Since $g_\omega\in \SLg$ is symmetric, its eigenvalues are real, nonzero, and have the same sign. Moreover, the map $\omega \to g_\omega$ is continuous on $(0,1]$, and $g_1 = \exp(\pi \xi_3) = -\Id$. Since $(0,1]$ is connected and the eigenvalues cannot cross zero, $g_\omega$ is negative-definite for every $\omega \ in (0,1]$. Finally, conjugation by \(Q_\phi\in \SO\) preserves symmetry and negative-definiteness. Therefore, every $g \in \MAS$ is symmetric and negative-definite.
\end{proof}


\section{Topological Study}\label{sec:top}
In this section, we complete the optimal synthesis and prove Theorem \ref{thm:main-3}. In Subsection \ref{subsec:had-the}, we verify the hypotheses of Hadamard's global diffeomorphism theorem for the restriction $\Exp\vcentcolon\MD\to \ND$, except for the topological properties of $\ND$. In Subsection \ref{subsubsec:proof-cons}, assuming these properties, we prove Theorem \ref{thm:main-3} and Corollaries \ref{cor:cut-time}-\ref{cor:equi-opt}. Finally, in Subsection \ref{subsec:N-simp-conn}, we analyze the topology of $\ND$ through the Maxwell stratum and establish Theorem \ref{thm:M-simp-conn}.


\subsection{Hadamard’s Global Diffeomorphism Theorem}\label{subsec:had-the}
We begin this section by presenting Hadamard’s Global Diffeomorphism Theorem \cite{krantz2002implicit}. 
\begin{backgroundtheorem}
    Let $M$ and $N$ be connected smooth manifolds of the same dimension, without boundary, and suppose that $\ND$ is simply connected. If $F\vcentcolon M \to N$ is a smooth, proper map such that $dF|_{p}$ is invertible for every $p\in M$, then $F$ is a global diffeomorphism.    
\end{backgroundtheorem}

In this section, we apply Hadamard’s Global Diffeomorphism Theorem to the exponential map $\Exp\vcentcolon\MD \to \ND$. By Corollary \ref{cor:t-max-t-con}, every nonzero initial velocity $\xi \in \MD$ determines a geodesic segment whose length is strictly smaller than its first conjugate time. Consequently, $d\Exp$ is nonsingular for every $\xi \in \MD$. Thus, we only need to verify the other hypotheses of Hadamard's theorem.

\begin{lemma}\label{lem:M-con}
  The subdomain  $\MD\subseteq\sla$ is open and contractible. 
\end{lemma}
\begin{proof}
    Recall that 
    $$ \MD \coloneqq \{0_{2}\}\cup \Big(\bigcup_{\omega \in [-1,1]}  \Psi\big((0,\Tmax(\omega))\times \Gamma_\omega\big) \Big). $$
    Since $\Tmax$ is continuous by Proposition \ref{prp:f-max-t-reg}, the set
    $$\bigcup_{\omega \in [-1,1]} (0,\Tmax(\omega))\times \Gamma_\omega$$
    is open in $\mathcal{E}$. Since $\Psi\vcentcolon\mathcal{E} \to \sla\setminus\{0_2\}$ is a diffeomorphism, its image is open in $\sla\setminus\{0_2\}$. Moreover, Proposition \ref{prp:f-max-t-reg} implies that the open ball of radius $\frac{\pi\sqrt{14}}{4}$ centered at the origin is contained in $\MD$. Therefore, $\MD$ is open.   Furthermore, the subdomain $\MD$ is star-shaped with respect to $0_2$. Indeed, if $\xi \in \MD$, then $(1-s)\xi \in \MD$ for all $s\in [0,1]$. Consequently,  the homotopy $H\vcentcolon\MD\times[0,1] \to \MD$ given by $H(\xi,s) = (1-s)\xi$ satisfies $H(\xi,0) = \xi$, and $H(\xi,1) = 0_2$.  Thus, $H$ defines a contraction of $\MD$ to $0_2$. In particular, $\MD$ is contractible and hence connected.
\end{proof}

\begin{lemma}\label{lem:exp-prp}
    The map $\Exp\vcentcolon\MD \to \ND$ is proper. 
\end{lemma}

\begin{proof}
   Recall that a continuous map $F\vcentcolon X\to Y$ between smooth manifolds is proper if and only if every sequence in $X$ that leaves every compact subset of $X$ is mapped to a sequence that leaves every compact subset of $Y$. 

    Suppose, for contradiction, that $\Exp\vcentcolon\MD\to\ND$ is not proper. Then there exist a compact set $K\subset \ND$ and a sequence $\{p_n\}\subseteq \MD$ leaving every compact subset of $\MD$ such that $\Exp(p_n) \in K$ for all $n$. Since $\{p_n\}$ leaves every compact subset of $\MD$, we may assume, after discarding finitely many terms, that $p_n\neq 0_2$ fro all $n$. Using polar coordinates, write
    $$ p_n = t_n\xi_n,\quad\xi_n \in \Gamma_{\omega_n}, \quad 0<t_n<\Tmax(\omega_n).$$
    Since $\mathbb{S}^2(\sla)$ is compact, after passing to a subsequence we may assume that $\xi_n\to \xi\in \Gamma_\omega$. In particular, $\omega_n \to \omega$.

    We distinguish two cases.

    \textbf{Case 1: $\{ t_n\}$ is bounded}. After passing to a subsequence, $t_n\to t$. Since \(\{p_n\}\) leaves every compact subset of $\MD$, its limit cannot belong to \(\MD\). Since $t_n<\Tmax(\omega_n)$ and $\Tmax$ is continuous by Proposition \ref{prp:f-max-t-reg}, we have $t\leq \Tmax(\omega)$. If $t<\Tmax(\omega)$, then $p_n \to t\xi \in \MD$, contradicting the assumption that $\{p_n\}$ leaves every compact subset of $\MD$. This includes the case $t= 0$, since $0_2\in\MD$. Therefore, $t=\Tmax(\omega)$. It follows that 
    $$\lim_{n\to \infty}\Exp\Big(t_n,\xi_n\Big) = \Exp\Big(\lim_{n\to \infty}(t_n,\xi_n)\Big) = \Exp(t_{\mathrm{Max}}(\omega),\xi)
\in\MAS.$$
    Since $\Exp(p_n) \in K$ and $K$ is compact, its limit belongs to $K\subset\ND$, contradicting $\MAS\cap\ND=\varnothing$.

     \textbf{Case 2: $\{t_n\}$ is unbounded}. After passing to a subsequence, $t_n \to \infty$. Since $\Tmax$ is continuous and finite on every compact subset of $[-1,1]\setminus \{0\}$, the condition $t_n<\Tmax(\omega)$, together with $t_n \to \infty$ implies that $\omega_n \to 0$. Hence, there exists $N>0$ such that $|\omega_n| < \frac{1}{2}$ for every $n>N$. By Lemma \ref{lem:expl-form}, for $n>N$,
    $$\Exp(p_n) =\left(\cosh(t_n \alpha_n) \Id + \frac{\sinh(t_n\alpha_n)}{\alpha_n} \xi^\T_n \right)\exp(2\omega_n t_n\xi_3), $$
    where $\alpha_n = \sqrt{1-2\omega_n^2}$.
    Since $\alpha_n \to 1$ as $\omega_n \to 0$, we have $t_n\alpha_n \to \infty$.  

    We extend the norm defined in Section \ref{sec:int} to the space of real $2\times2$ matrices by setting
    $$ \|A\|^2\coloneqq\frac{1}{2}\Tr(AA^\T).$$ 
    Define $Q_n\coloneqq\exp(\omega_n t_n\xi_3)$, since the right multiplication by $Q_n$, preserves the norm. Hence,
    $$ \| \Exp(t_n,\xi_n)\| = \left\|\cosh(t_n\alpha_n) Id + \frac{\sinh(t_n\alpha_n)}{\alpha_n} \xi^\T_n  \right\|.$$
    Since $\Tr(\xi_n) = 0$, the matrices $Id$ and $\xi_n^\T$ are orthogonal with respect to the inner product. Thus,
    $$\| \Exp(t_n,\xi_n)\|^2  \geq \cosh^2( t_n\alpha_n).$$
    Since $t_n\alpha_n \to \infty$, it follows that $\| \Exp(t_n,\xi_n)\| \to \infty$. Therefore, $\{\Exp(t_n,\xi_n)\}$ leaves every compact set in  $\ND$. 
\end{proof}


\begin{Theorem}\label{thm:M-simp-conn}
    The subset $\ND$ is open and diffeomorphic to $\R^3$. In particular, it is open and simply connected. 
\end{Theorem}

The proof of Theorem \ref{thm:M-simp-conn} is the most technical part of this section, and we postpone it to the next subsection. For now, assuming Theorem \ref{thm:M-simp-conn}, we prove Theorem \ref{thm:main-3} and derive its consequences.

\subsubsection{Proof of Theorem \ref{thm:main-3} and Its Consequences}\label{subsubsec:proof-cons}

    Lemmas \ref{lem:M-con} and \ref{lem:exp-prp}, together with Theorem \ref{thm:M-simp-conn}, imply that $\Exp\vcentcolon \MD \to \ND$ is a diffeomorphism. This proves Theorem \ref{thm:main-3}.

\begin{proof}[Proof of Corollary \ref{cor:cut-time}]
   As a geodesic cannot remain minimizing beyond its first Maxwell point, $\Tcut (\omega)\leq \Tmax(\omega)$. Conversely, by Theorem \ref{thm:main-3} and the characterization of the Maxwell stratum, no geodesic loses optimality before its first Maxwell time. Hence, $\Tcut (\omega)\geq \Tmax(\omega)$. Therefore, $\Tcut (\omega) = \Tmax(\omega)$. Consequently, $\mathcal{Cut}(\Id) = \MAS$. For $\omega = 0$, $\Tmax(0) = \infty$, so the corresponding geodesics are globally minimizing.
\end{proof}

\begin{proof}[Proof of Corollary \ref{cor:rad-in}]
    By Corollary \ref{cor:cut-time} and Proposition \ref{prp:f-max-t-reg},
    $$ \operatorname{inj}(\Id) = \min_{\omega\in [-1,1]} \Tcut (\omega) = \min_{\omega\in [-1,1]} \Tmax(\omega) = \Tmax(\frac{3}{\sqrt{14}}) = \frac{\pi \sqrt{14}}{4}. $$
\end{proof}

\begin{proof}[Proof of Corollary \ref{cor:cut-luc-inter}]
 By Corollary \ref{cor:cut-time} and Theorem \ref{thm:classic},
 $$\mathcal{Cut}(\Id) \cap \mathcal{Con}(\Id) = \MAS \cap \SO.$$
 By Lemma \ref{lem:g-ome-sym}, every $g \in \MAS$ is symmetric and negative-definite. The only symmetric negative-definite matrix in $\SO$ is $-\Id$. Moreover, $-\Id \in \MAS$. Indeed, since $\Tmax(\pm 1) = \pi$, Eq. \eqref{eq:def-gam-ome} and Lemma \ref{lem:expl-form} imply $\Exp(\pi, \xi_{\pm 1}) = - \Id$. Therefore, 
  $$\mathcal{Cut}(\Id) \cap \mathcal{Con}(\Id) =\{ - \Id\}.$$
\end{proof}

\begin{proof}[Proof of Corollary \ref{cor:equi-opt}]
    Let $\gamma_\omega$ be a geodesic with $\gamma_\omega(0) \in \Gamma_\omega$, and fix $t_0\in \R$. Consider the geodesic $\gamma^1_\omega(t) \coloneqq \gamma_\omega^{-1}(t_0) \gamma_\omega(t+t_0)$. By left-invariance, $\gamma^1_\omega$ is a geodesic starting at the identity. Moreover, $\dot{\gamma}^1_\omega(0) \in \Gamma_\omega$. Hence, by Corollary \ref{cor:cut-time}, 
    $$ \Tcut \big(\Id,\dot{\tilde{\gamma}}_\omega(0)\big) = \Tcut \big(\omega\big). $$
    Since left translations are isometries,
    $$t_{\mathrm{cut}}\big(\gamma_\omega(t_0),\dot{\gamma}_\omega(t_0)\big) = t_{\mathrm{cut}}\big(\Id,\dot{\tilde{\gamma}}_\omega(0)\big) = t_{\mathrm{cut}}(\omega).$$
    Thus, $t_{\mathrm{cut}}\big(\gamma_\omega(t_0),\dot{\gamma}_\omega(t_0)\big)$ is independent of \(t_0\), and therefore \(\gamma_\omega\) is equi-optimal.
\end{proof}

\subsection{Proof of Theorem \ref{thm:M-simp-conn}}\label{subsec:N-simp-conn}

Our approach follows from the Iwasawa decomposition \cite{lang1985sl}. Every $g \in \SLg$ admits a unique factorization 
$$ g = ank\footnote{The decomposition $ANK$ is a variant of the standard \(KAN\) Iwasawa decomposition.}$$
where
\begin{equation}\label{eq:iwa-des}
    \begin{split}
        a = \left[\begin{array}{cc}
    e^s & 0 \\
    0 & e^{-s}
\end{array}\right],\qquad s\in \R,\\
        n = \left[\begin{array}{cc}
    1 & x \\
    0 & 1
\end{array}\right],\qquad x\in \R,\\
    \end{split}
\end{equation}
and
\begin{equation}\label{eq:iwa-des-2}
    \begin{split}
        k = \left[\begin{array}{cc}
    \cos\phi & -\sin\phi \\
    \sin\phi & \cos\phi
\end{array}\right] , \qquad \phi\in \R/2\pi\mathbb{Z}.\\
    \end{split}
\end{equation}
Moreover, the multiplication map
$$ A\times N\times K \to \SLg,\quad (a,n,k) \mapsto ank$$
is a diffeomorphism. In other words, the Iwasawa decomposition identifies $\SLg$ with $\R^2\times \mathbb{S}^1$, where the $AN$-factor corresponds to $\R^2$ and the $K$-factor corresponds to $\mathbb{S}^1 \cong \SO$.

The strategy is as follows. In Lemma \ref{lem:sym-ide}, we identify the factor \(AN\) with $\Sym$, the space of positive-definite symmetric \(2\times2\) matrices of determinant one. We then introduce global coordinates \((\varphi,\vartheta)\in\mathbb R^2\) on $\Sym$ and describe the conjugation action of  \(\SO\) in these coordinates. In Proposition \ref{prp:bij-mas-sym}, we show that the projection of the Maxwell stratum onto the $(\varphi,\vartheta)$-plane is a diffeomorphism. Consequently, the Maxwell stratum can be represented as the graph of a function $\Theta\vcentcolon\R^2 \to K$. Finally, we use this graph representation to prove that $\ND$ is simply connected. 

\subsubsection{The Set of Positive-definite Symmetric Matrices in \(\SLg\)}

We begin by formally defining the set $\Sym$ as
$$\Sym \coloneqq \{ S\in\SLg\vcentcolon S=S^\T,\;\;S \text{ is positive definite}\}. $$

\begin{lemma}\label{lem:sym-ide}
The map $\F\vcentcolon\SLg \to \Sym$ given by $\F(g) = gg^\T$ is well-defined. Furthermore,  the restriction of $\F|_{AN}$ is a diffeomorphism.
\end{lemma}

\begin{proof}
    Since $g$ is invertible, $gg^\T$ is symmetric and positive definite. Moreover, $\det(gg^\T) = 1$. Thus, $\F$ is well-defined. Moreover, if $g=ank$ then $\F(g) = ann^\T a$, so $\F(g)$ is independent of the $k$-factor. To express the image of $\F$ in coordinates, we consider $g=an$, where $a$ and $n$ have the form of Eq. \eqref{eq:iwa-des}. It follows that
    $$\F(an)  =\left[\begin{array}{cc}
    e^{2s}(1+x^2) & x \\
    x & e^{-2s}
\end{array}\right].$$

Consider the global coordinates $(\varphi,\vartheta)\in \R^2$ on $\Sym$ and write an element $S$ as   
\begin{equation}\label{eq:def-S}
    S = \left[\begin{array}{cc}
    \sqrt{1+\varphi^2+\vartheta^2}+\varphi & \vartheta \\
    \vartheta & \sqrt{1+\varphi^2+\vartheta^2} - \varphi
\end{array}\right]. 
\end{equation}
Then we have \((e^{-2s},x) = (\sqrt{1+\varphi^2+\vartheta^2} - \varphi, \vartheta)\), i.e.,
\[(s,x)=\left(-\frac{1}{2}\log(\sqrt{1+\varphi^2+\vartheta^2} - \varphi),\vartheta\right),\] 
which determines a unique element $an\in AN$. Hence, the inverse of $\F|_{AN}$ is a diffeomorphism.
\end{proof}

Consequently, $(\F|_{AN})^{-1}\vcentcolon \Sym \to AN$ induces global coordinates on $AN$.

We now consider the conjugation action $\Phi_Q$ of $\SO$ on $\SLg$ and the induced action $\Tilde{\Phi}_Q$ on $\Sym$. Thus $\F$ is an $\SO$-equivariant map. Indeed, if $g\in \SLg$ then
$$ \F(\Phi_Q(g)) = \F(Q^\T gQ) = Q^\T g g^\T Q = \Tilde{\Phi}_Q(\F(g)).  $$
Moreover, if $Q_{\phi} \in \SO$ is given by Eq. \eqref{eq:iwa-des-2}, then the $\SO$-action on $\Sym$ takes the form of a rotation by $-2\phi$ on  the $(\varphi,\vartheta)$-plane. Indeed, if $S\in \Sym$ has the form from Eq. \eqref{eq:def-S}, then a direct computation shows that 
$$ \Tilde{\Phi}(S) = \left[\begin{array}{cc}
    \sqrt{1+\varphi^2+\vartheta^2}+ \tilde{\varphi} & \tilde{\vartheta} \\
    \tilde{\vartheta} & \sqrt{1+\varphi^2+\vartheta^2} - \tilde{\varphi}
\end{array}\right].  $$
Here, $(\tilde{\varphi},\tilde{\vartheta}) = (\varphi \cos(2\phi)+\vartheta\sin(2\phi),-\varphi \sin(2\phi) + \vartheta\cos(2\phi))$.
Therefore, the squared radius $\varphi^2+\vartheta^2$ is invariant under the $\SO$-action on $\Sym$.

Define $g_\omega = \gamma_\omega(\Tmax)$. It follows that 
$$ \MAS = \big\{ \Phi_{Q_\phi}(g_\omega)  \vcentcolon \omega \in (0,1)\;\;\&\;\; \phi\in [0,\pi) \big\} \cup \{ -\Id\}.$$
Thus, the orbit of $\F(g_\omega)$ by the $\SO$-action is a circle on the $(\varphi,\vartheta)$-plane, whose radius only depends on $\omega$. We denote this radius by $\rho\vcentcolon(0,1) \to [0,\infty)$. To prove that the projection of the Maxwell stratum onto the $(\varphi,\vartheta)$-plane is a diffeomorphism, we study the function $\rho$. Now instead of working with $\rho$, we define an auxiliary function $q\vcentcolon(0,1)\to [0,\infty)$. 

From Eq. \eqref{eq:def-S}, we have 
$$ \Tr(S) = 2\sqrt{1+\varphi^2+\vartheta^2} \;\;\Rightarrow\;\; \varphi^2+\vartheta^2 = \frac{1}{4}\Tr(S)^2-1 .$$
We define $q$ by
\begin{equation}\label{eq:def-q}
    \Tr(\F(g_\omega)) \coloneqq 2+q^2(\omega) \;\;\Rightarrow\;\; \rho^2(\omega)=q^2(\omega)\left(1+\frac{1}{4}q^2(\omega)\right).
\end{equation}
The function $q\vcentcolon(0,1)\to [0,\infty)$ is well-defined since $\Tr(S)\geq 2$. Moreover, \(\rho\) and \(q\) have the same monotonicity.

\begin{Prop}\label{prp:bij-mas-sym}
    The function $q\vcentcolon(0,1)\to [0,\infty)$ defined in Eq. \eqref{eq:def-q} satisfies
    \[
q'(\omega)<0,\qquad 0<\omega<1,
\]
and
\[
\lim_{\omega\to0^+}q(\omega)=+\infty,
\qquad
\lim_{\omega\to1^-}q(\omega)=0.
\]
Consequently, $\rho\vcentcolon(0,1)\to(0,\infty)$ is a strictly decreasing bijection and $\F\vcentcolon\MAS \to \Sym$ is a diffeomorphism. 
\end{Prop}

\subsubsection{Proof of Proposition \ref{prp:bij-mas-sym}}
We begin by computing the function $q$ explicitly for the different regimes of $\omega$. 

\begin{lemma}
The function $q\vcentcolon(0,1) \to (0,\infty)$ has the form
\begin{equation}\label{eq:q-def}
    q(\omega)=
\begin{cases}
\displaystyle
\frac{2\sqrt{1-\omega^2}}{\sqrt{2\omega^2-1}}\,

\sin\!\left(\Tmax(\omega)\sqrt{2\omega^2-1}\right)
,
&\;\;\text{if} \;\; \omega \in (\frac{1}{\sqrt{2}},1),
\\[4mm]
\displaystyle
\sqrt{2}\,\Tmax(\omega),
& \;\;\text{if} \;\; \omega=\dfrac1{\sqrt2},
\\[4mm]
\displaystyle
\frac{2\sqrt{1-\omega^2}}{\sqrt{1-2\omega^2}}\,
\sinh\!\left(\Tmax(\omega)\sqrt{1-2\omega^2}\right),
& \;\;\text{if} \;\; \omega\in(0,\frac{1}{\sqrt{2}}).
\end{cases}
\end{equation}    
\end{lemma}

\begin{proof}
Consider $\omega \in (\frac{1}{\sqrt{2}},1)$. Recall the explicit formula for $g_{\omega}\coloneqq\gamma_\omega(\Tmax(\omega))$ from Lemma \ref{lem:expl-form}, and the vector $\xi_\omega\in \sla$ defined in Eq. \eqref{eq:xi-ome}. For simplicity, set $\alpha=\sqrt{2\omega^2-1}$ and $\beta = \Tmax(\omega)\alpha$. Then
\begin{equation*}
\begin{split}
    \Tr(\F(g_\omega)) & = \Tr\!\Big(\cos^2\beta \Id + \frac{2}{\alpha}\cos\beta\sin\beta\sqrt{1-\omega}(\xi_\omega+\xi_\omega^\T) + \frac{\sin^2\beta}{\alpha^2}\xi_\omega\xi_\omega^\T \Big)  \\
    & = 2\Big(\cos^2\beta +  \frac{\sin^2\beta}{\alpha^2}\Big) = 2 +2\sin^2\beta(\frac{1}{\alpha^2}-1) = 2+\frac{2\sin^2\beta}{\alpha^2}(1-\alpha^2)\\
    & = 2 +\frac{4(1-\omega^2)}{2\omega^2-1}\sin\!\left(\Tmax(\omega)\sqrt{2\omega^2-1}\right) . 
\end{split}
\end{equation*}
We remark that for $\omega \in (\frac{1}{\sqrt{2}},1)$ we have $\Tmax(\omega)\sqrt{2\omega^2-1}\in (0,\pi)$. Hence, the sine is positive, and Eq. \eqref{eq:def-q} yields the desired expression. The remaining cases follow the same strategy. 
\end{proof}

To prove Proposition \ref{prp:bij-mas-sym}, we show the properties of the function $q$ in the following lemmas, one for each regime of $\omega$.

\begin{lemma}\label{lem:q-prp-1}
    For $\omega \in (\frac{1}{\sqrt{2}},1)$, the function $q(\omega)$ satisfies $\frac{dq}{d\omega}<0$. 
\end{lemma}

\begin{proof}
    Consider the case $\omega\in (\frac{1}{\sqrt{2}},1)$. Let us define the auxiliary functions,
    $$ \alpha(\omega)\coloneqq\sqrt{2\omega^2-1},\;\;\beta(\omega)\coloneqq \alpha (\omega) \Tmax(\omega),\;\; z(\omega)\coloneqq \tan \beta(\omega).  $$
    Then we rewrite $q$ as
    \begin{equation}\label{eq:q-omega-1}
        q(\omega) = \frac{2\sqrt{1-\omega^2}}{\alpha(\omega)} \sin \beta(\omega).
    \end{equation}
    For $\omega \in (\frac{1}{\sqrt{2}},1)$ we have $0<\beta(\omega)<\pi$, so $\sin \beta(\omega)>0$. If $\beta(\omega) = \frac{\pi}{2}$, it follows that 
    $$\frac{dq}{d\omega} = \frac{d}{d\omega}\Big(\frac{2\sqrt{1-\omega^2}}{\alpha(\omega)}\Big) = - \frac{2\omega}{\sqrt{1-\omega^2}\alpha^3}<0.$$ 
    Therefore, we proceed with the case $\beta(\omega) \neq \frac{\pi}{2}$. 
    We rewrite the expression for $\frac{dq}{d\omega}$ in Eq. \eqref{eq:dt-mx-domega-1} in terms of the auxiliary functions
    $$ \frac{d}{d\omega} \Tmax = \frac{2\alpha(1-\omega^2)\Tmax(\omega)-z}
{\alpha\omega(\alpha^2+z^2)}.$$

Since $q(\omega)>0$, logarithmic differentiation of Eq. \eqref{eq:q-omega-1}  gives
$$\frac{1}{q}\frac{dq}{d\omega} = -\frac{\omega}{1-\omega^2} - \frac{1}{\alpha(\omega)} \frac{d\alpha}{d\omega} + \cot \beta(\omega) \frac{d\beta}{d\omega}.$$
The definitions of $\alpha(\omega)$, $\beta(\omega)$, and $z(\omega)$ give
$$ \frac{d\alpha}{d\omega} = \frac{2\omega}{\alpha},\;\;\; \frac{d\beta}{d\omega} = \frac{2\omega}{\alpha}\Tmax(\omega) + \alpha(\omega)\frac{d}{d\omega} \Tmax(\omega),\;\; \cot \beta = \frac{1}{z}.  $$
Therefore,
\begin{equation}\label{eq:dq-dome}
    \begin{split}
       \frac{1}{q}\frac{dq}{d\omega} & = - \frac{\omega}{1-\omega^2} - \frac{2\omega}{\alpha^2(\omega)}+\frac{1}{z}\Big(\frac{2\omega}{\alpha}\Tmax(\omega) + \alpha(\omega)\frac{d}{d\omega} \Tmax(\omega)\Big)\\
       & =\frac{
(\alpha^2+\omega^2z^2)
\left[
2\alpha(1-\omega^2)\Tmax(\omega)-z
\right]
}
{
\omega z(1-\omega^2)\alpha^2(\alpha^2+z^2)
}. 
    \end{split}
\end{equation}

In Eq. \eqref{eq:dq-dome}, the terms are positive except for the expression
$$\frac{2\alpha(1-\omega^2)\Tmax(\omega)-z}{z} = \frac{2(1-\omega^2)\beta-\tan\beta}{\tan\beta}. $$
Thus, it suffices to prove that the above expression is negative. We proceed by the cases $\beta\in (0,\frac{\pi}{2})$ and $\beta\in (\frac{\pi}{2},\pi)$.

\textit{Case $\beta\in (0,\frac{\pi}{2})$.} In this case, we have $\tan \beta(\omega)>\beta(\omega)>0$ and $2(1-\omega^2) < 1$. Therefore
\begin{equation*}
    \begin{split}
        2(1-\omega^2)\beta(\omega)<\tan\beta(\omega) \;\;\Rightarrow\;\; 2(1-\omega^2)\beta(\omega) - \tan\beta(\omega)<0. 
    \end{split}
\end{equation*}
Hence $\frac{dq}{d\omega}<0$.

\textit{Case $\beta\in (\frac{\pi}{2},\pi)$.} In this case, we have $z=\tan\beta(\omega) < 0$ and $2(1-\omega^2)\beta(\omega)>0$. Thus
$$ 2(1-\omega^2)\beta(\omega) - \tan\beta(\omega)>0.$$
Dividing by $\tan \beta(\omega)$ gives $\frac{dq}{d\omega}<0$.
Thus, in either case $\frac{dq}{d\omega}<0$. \end{proof}

\begin{lemma}\label{lem:q-prp-2}
    For $\omega \in (0,\frac{1}{\sqrt{2}})$, the function $q(\omega)$ satisfies $\frac{dq}{d\omega}<0$. 
\end{lemma}

\begin{proof}
    Consider $\omega \in (0,\frac{1}{\sqrt{2}})$. Let
    $$ \alpha(\omega)\coloneqq\sqrt{1-2\omega^2},\;\;\beta(\omega)\coloneqq \alpha (\omega) \Tmax(\omega),\;\; h(\omega)\coloneqq \tanh \beta(\omega).  $$
    Then we rewrite $q(\omega)$ as
    \begin{equation}\label{eq:q-omega-2}
        q(\omega) = \frac{2\sqrt{1-\omega^2}}{\alpha(\omega)} \sinh \beta(\omega).
    \end{equation}
    We rewrite the expression for $\frac{dq}{d\omega}$ in Eq. \eqref{eq:dt-mx-domega-2} in terms of the auxiliary functions
    $$ \frac{d}{d\omega} \Tmax = \frac{h(\omega)-2\alpha(\omega)(1-\omega^2)\Tmax(\omega)}
{\alpha(\omega)\omega(\alpha^2(\omega)+h^2(\omega))}.$$

Now compute the logarithmic differentiation for Eq. \eqref{eq:q-omega-2}  to obtain
$$\frac{1}{q}\frac{dq}{d\omega} = -\frac{\omega}{1-\omega^2} - \frac{1}{\alpha(\omega)} \frac{d\alpha}{d\omega} + \coth \beta(\omega) \frac{d\beta}{d\omega}.$$
The definition of $\alpha(\omega)$, $\beta(\omega)$, and $h(\omega)$ give
$$ \frac{d\alpha}{d\omega} = -\frac{2\omega}{\alpha(\omega)},\;\;\; \frac{d\beta}{d\omega} = -\frac{2\omega}{\alpha(\omega)}\Tmax(\omega) + \alpha(\omega)\frac{d}{d\omega} \Tmax(\omega),\;\; \coth \beta(\omega) = \frac{1}{h}.  $$
Therefore,
\begin{equation}\label{eq:dq-dome-1}
    \begin{split}
       \frac{1}{q}\frac{dq}{d\omega} & =  -\frac{\omega}{1-\omega^2} + \frac{2\omega}{\alpha^2(\omega)}+\frac{1}{h(\omega)}\Big(-\frac{2\omega}{\alpha(\omega)}\Tmax(\omega) + \alpha(\omega)\frac{d}{d\omega} \Tmax(\omega)\Big)\\
       & =\frac{(\alpha^2(\omega)+\omega^2h^2(\omega))\left[h(\omega)-
2\alpha(\omega)(1-\omega^2)\Tmax(\omega)
\right]
}
{
\omega h(\omega)(1-\omega^2)\alpha^2(\omega)(\alpha^2(\omega)+h^2(\omega))
}. 
    \end{split}
\end{equation}
In Eq. \eqref{eq:dq-dome-1}, the terms are positive except for the expression
$$h-2\alpha(1-\omega^2)\Tmax(\omega) = \tanh\beta-2(1-\omega^2)\beta. $$
Thus, it suffices to prove that the above expression is negative. It follows that $2(1-\omega^2) > 1$, and since $u>0$, we have $\tanh u<u$. Therefore, 
$$ \tanh \beta(\omega)< \beta(\omega) < 2(1-\omega^2) \beta(\omega).$$
Hence $\frac{dq}{d\omega}<0$. 
\end{proof}

\begin{lemma}\label{lem:q-prp-3}
    The function $q(\omega)$ satisfies $\frac{dq}{d\omega}<0$ at $\omega = \frac{1}{\sqrt{2}}$. 
\end{lemma}
\begin{proof}
    Define $s\coloneqq2\omega^2-1$. Then $s=0$ correspons to $\omega = \frac{1}{\sqrt{2}}$. 
    \begin{equation*}
        \begin{split}
            q(\omega) & = 2\sqrt{1-\omega^2} \frac{\sin(\Tmax(\omega)\sqrt{s})}{\sqrt{s}},\qquad s>0,\\
            q(\omega) &= 2\sqrt{1-\omega^2} \frac{\sinh(\Tmax(\omega)\sqrt{-s})}{\sqrt{-s}},\qquad s<0.
        \end{split}
    \end{equation*}
    Both expressions have the same expansion:
    \begin{equation*}
        \begin{split}
            \frac{\sin(t\sqrt{s})}{\sqrt{s}} & = t-\frac{t^3}{6}s +O(s^2),\qquad s>0,\\
            \frac{\sinh(t\sqrt{-s})}{\sqrt{-s}} &=  t-\frac{t^3}{6}s +O(s^2) ,\qquad s<0.
        \end{split}
    \end{equation*}
    Thus, on both sides,
$$q(\omega) = 2\sqrt{1-\omega^2}
\left[ \Tmax(\omega)-\frac{\Tmax^3(\omega)}{6}(2\omega^2-1) +O((2\omega^2-1)^2)\right].
$$
Hence
$$ \frac{dq}{d\omega}\big|_{\omega=\frac{1}{\sqrt{2}}} = (-2\Tmax(\omega) + \sqrt{2}\frac{d}{d\omega}\Tmax -\frac{2}{3}\Tmax^3(\omega))|_{\omega=\frac{1}{\sqrt{2}}} $$
In the proof of Proposition \ref{prp:f-max-t-reg}, we show that $\frac{d}{d\omega}\Tmax<0$ at $\omega= \frac{1}{\sqrt{2}}$. Combining this result with the above equation gives the desired result.
\end{proof}

\begin{proof}[Proof of Proposition \ref{prp:bij-mas-sym}]
   Lemmas \ref{lem:q-prp-1}, \ref{lem:q-prp-2}, and \ref{lem:q-prp-3} show $\frac{dq}{d\omega}< 0$ for every $\omega\in(0,1)$.   It remains to verify the limiting behavior of \(q\) at the endpoints. Recall that the condition $\Tmax(\omega)\in (\frac{\pi}{2\omega},\frac{\pi}{\omega})$ implies that 
   \[
\lim_{\omega\to0^+}\Tmax(\omega)\sqrt{1-2\omega^2}=+\infty,
\]
then
   \[
\lim_{\omega\to 0^+}\frac{2\sqrt{1-\omega^2}}{\sqrt{1-2\omega^2}}\,
\sinh\!\big(\Tmax(\omega)\sqrt{1-2\omega^2}\big)= 2\lim_{\omega\to 0^+}
\sinh\!\big(\Tmax(\omega)\sqrt{1-2\omega^2}\big) =\infty.
\]
For $\omega\in(\frac{1}{\sqrt{2}},1)$, we have
$$ 0 \leq q(\omega) \leq \frac{2\sqrt{1-\omega^2}}{\sqrt{2\omega^2-1}}, \quad \text{and} \quad \lim_{\omega\to 1^-}\frac{2\sqrt{1-\omega^2}}{\sqrt{2\omega^2-1}}=0.$$
By the squeeze theorem, we obtain the desired result. 

Since \(\rho\) and \(q\) have the same monotonicity, we conclude that $\rho$ is strictly decreasing. Moreover, for \(\omega\in(0,1)\), the radial-angular argument implies
$$ \F\vcentcolon\MAS\setminus\{-\Id\} \to \Sym\setminus\{\Id\}$$
is a diffeomorphism, and the second limit above implies that $\F$ extends continuously to $g_\omega|_{\omega=1}$. Lemma \ref{lem:g-ome-sym} implies that if $S\in \F(\MAS)$ then $(\F|_{\MAS})^{-1}(S) = -S^{\frac{1}{2}}$. Since the positive-definite square-root map is smooth, the inverse of \(\F|_{\MAS}\) is smooth. Therefore, \(\F\) is a diffeomorphism.
\end{proof}

\subsubsection{Proof of Theorem \ref{thm:M-simp-conn}}
The following lemma is the polar decomposition of $\SLg$ adapted to our notation
\begin{lemma}
    The map 
$$ \Pm\vcentcolon\SLg\to\Sym \times \SO,\qquad \Pm(g) = (\F(g),(\F(g))^{-\frac{1}{2}}g),$$
is a diffeomorphism, with inverse
$$ \Pm^{-1}(S,Q) = S^{\frac{1}{2}}Q.$$
\end{lemma}
\begin{proof}
    By definition of $\F\vcentcolon\SLg \to \Sym$, the first component of $\Pm$ is well defined. Set $\F(g) = S$ and $Q\coloneqq S^{-\frac{1}{2}}g$. Then 
    $$QQ^\T  =  S^{-\frac{1}{2}}gg^\T S^{-\frac{1}{2}} = S^{-\frac{1}{2}}S S^{-\frac{1}{2}} = \Id.$$
    Moreover, $\det(Q)=\det(S^{-\frac{1}{2}})\det(g) = 1$. Thus, $Q\in \SO$ and $g=S^{\frac{1}{2}}Q$. 

    Conversely, let $S\in \Sym$ and $Q\in \SO$. Then
    $$ (S^{\frac{1}{2}}Q) (S^{\frac{1}{2}}Q)^\T = S^{\frac{1}{2}}QQ^\T S^{\frac{1}{2}} = S.$$
    Therefore $\Pm(S^{\frac{1}{2}}Q) = (S,Q)$ and hence $\Pm^{-1}(S,Q) =S^{\frac{1}{2}}Q. $
    Since the maps $S \mapsto S^{\frac{1}{2}}$ and $S \mapsto S^{-\frac{1}{2}}$ are smooth on the space of positive-definite symmetric matrices, both $\Pm$ and $\Pm^{-1}$ are smooth.  Hence, $\Pm$ is a diffeomorphism.
\end{proof}

\begin{proof}[Proof of Theorem \ref{thm:M-simp-conn}]
    By Proposition \ref{prp:bij-mas-sym} and Lemma \ref{lem:g-ome-sym}, we have
    $$ \Pm(\MAS) = \Sym \times \{-\Id\}.$$
    Therefore
    $$ \ND \coloneqq \SLg\setminus\MAS \simeq \Sym \times (\SO\setminus\{-\Id\}) \simeq \R^2\times\R \simeq \R^3.$$
    Hence, $\ND$ is open and simply connected. 
\end{proof}

\section*{Acknowledgments}
The authors thank the Department of Mathematics at the University of Michigan for providing a supportive environment for this research. This work originated as a project in the Geometry Lab at the University of Michigan. We would like to acknowledge Harish Jaisankar, Ziming Sun, and George Li, who participated in the initial stages of the project but were unable to continue their involvement for various reasons.

The authors are also grateful to Richard Montgomery and Andrei Ardentov for their valuable email discussions and insights, regardless of the research's eventual outcome.
The second author wishes to express his gratitude to Velimir Jurdjevic, Andrei Agrachev, Yuri Sachkov, and Andrei Ardentov, whose papers and books provided an invaluable introduction to the theory of optimal synthesis.

\bibliographystyle{plain}
\bibliography{bibli.bib}

\end{document}